\documentclass[12pt]{amsart}
\usepackage{amssymb}
\newcommand{\h}{\mathfrak{h}}
\newcommand{\Hom}{\operatorname{Hom}}
\newcommand{\Ext}{\operatorname{Ext}}
\newcommand{\B}{\mathcal{B}}

\newcommand{\End}{\operatorname{End}}
\newcommand{\param}{\mathfrak{p}}
\newcommand{\C}{\mathbb{C}}
\newcommand{\Z}{\mathbb{Z}}
\newcommand{\Q}{\mathbb{Q}}

\newcommand{\gr}{\operatorname{gr}}

\newcommand{\HC}{\operatorname{HC}}
\newcommand{\VA}{\operatorname{V}}

\newcommand{\Leaf}{\mathcal{L}}
\newcommand{\A}{\mathcal{A}}

\newcommand{\J}{\mathcal{J}}

\newcommand{\Pro}{\mathcal{P}}
\newcommand{\Coh}{\operatorname{Coh}}
\newcommand{\F}{\mathbb{F}}
\newcommand{\Fr}{\operatorname{Fr}}

\newcommand{\Weyl}{\mathbb{A}}
\newcommand{\R}{\mathbb{R}}
\newcommand{\D}{\mathcal{D}}
\newtheorem{Thm}{Theorem}[section]
\newtheorem{Prop}[Thm]{Proposition}
\newtheorem{Cor}[Thm]{Corollary}
\newtheorem{Lem}[Thm]{Lemma}
\theoremstyle{definition}

\newtheorem{defi}[Thm]{Definition}
\newtheorem{Rem}[Thm]{Remark}
\newtheorem{Conj}[Thm]{Conjecture}
\numberwithin{equation}{section}
\address{Department
of Mathematics, University of Toronto, ON, Canada \&
National Research University Higher School of Economics, Russian Federation}
\email{ivan.loseu@gmail.com}
\title{Derived equivalences for Symplectic reflection algebras}
\author{Ivan Losev}
\begin{document}
\begin{abstract}
In this paper we study derived equivalences for symplectic reflection algebras. We establish a version
of  the derived localization theorem between categories of modules over these algebras
and categories of coherent sheaves over quantizations of $\Q$-factorial terminalizations of the
symplectic quotient singularities. To do this we construct a Procesi sheaf on the terminalization
and show that the quantizations of the terminalization are simple sheaves of algebras. We will also sketch
some applications: to the generalized Bernstein inequality and to perversity of wall crossing functors.
\end{abstract}
\maketitle
\section{Introduction}
The goal of this paper is to investigate derived equivalences between categories of modules
over symplectic reflection algebras (these algebras were introduced by Etingof and Ginzburg in
\cite{EG}) and give some  applications.

Let us briefly recall what these algebras are.  Let $V$ be a finite dimensional symplectic vector space over $\C$. Let $\Gamma$ be a finite subgroup of  $\operatorname{Sp}(V)$. Then we can form the smash-product (a.k.a. semi-direct product) algebra $\C[V]\#\Gamma$ that carries a natural grading. A symplectic reflection algebra $H_c$ is a filtered deformation of $\C[V]\#\Gamma$ (a note for experts: in this paper we only consider deformations with $t=1$ that should be thought as ``quantizations'' of $\C[V]\#\Gamma$). Here $c$ denotes a deformation parameter. Namely, consider the
set $S$ of all {\it symplectic reflections} in $\Gamma$, i.e., the  elements $s\in \Gamma$ such that $\operatorname{rk}(s-1)=2$. Let $\param$ denote the space of all conjugation invariant functions $S\rightarrow \C$. We pick $c\in \param$.
We will recall the definition of $H_c$ in Section \ref{SS_SRA}.

Consider the subalgebra $\C\Gamma\subset H_c$ and the averaging idempotent $e\in \C\Gamma$. The {\it spherical subalgebra}
$eH_ce$ is a quantization of $\C[V]^\Gamma$. The variety $V/\Gamma$ is a conical symplectic singularity. Consider
its $\Q$-factorial terminalization $X$ (see Proposition \ref{Prop:Q_fac_terminalization}).
We can talk about filtered quantizations of $\mathcal{O}_X$, see, e.g.,
\cite[Section 3]{BPW}. These are sheaves of filtered algebras on $X$ (in the so called {\it conical topology}).
The filtered quantizations of $X$ are parameterized by $H^2(X^{reg},\C)$, see \cite[Sections 3.1,3.2]{BPW}. For
$\lambda\in H^2(X^{reg},\C)$ we write $\mathcal{D}_\lambda$ for the corresponding quantization
of $\mathcal{O}_X$.
Moreover, in \cite[Section 3.7]{orbit}, we have established an affine isomorphism
$\param\xrightarrow{\sim} H^2(X^{reg},\C), c\leftrightarrow \lambda$, such that
$eH_ce\cong \Gamma(\mathcal{D}_\lambda)$.

Here is the first main result of this paper (that would be standard if $X$ were smooth but, for most $\Gamma$,
the variety $X$ is singular).

\begin{Thm}\label{Thm:simplicity}
The sheaf of algebras $\mathcal{D}_\lambda$ is simple for any $\lambda\in H^2(X^{reg},\C)$.
\end{Thm}

Our next result in this paper proves a conjecture from \cite[Section 7.1]{rouq_der}.
We say that parameters $c,c'\in \param$ have {\it integral difference} if their images in $H^2(X^{reg},\C)$
lie in the image of $\operatorname{Pic}(X^{reg})$.

\begin{Thm}\label{Thm:der_equiv1}
Let $c,c'\in \param$ have integral difference. Then there is a derived equivalence
$D^b(H_c\operatorname{-mod})\xrightarrow{\sim} D^b(H_{c'}\operatorname{-mod})$.
\end{Thm}

This theorem is proved using a strategy from \cite[Section 5]{GL}, where the special case of the theorem in the case of wreath-product groups $\Gamma$ was proved.
Namely, we construct a certain sheaf $\mathcal{P}$ on $X$ that we call a {\it Procesi sheaf} that generalizes the notion of a Procesi bundle in the case when $X$ is smooth, see \cite{BK, Procesi}.
Then we quantize $\mathcal{P}$ to a right $\mathcal{D}_c$-module
denoted by $\mathcal{P}_c$  (we write $\mathcal{D}_c$ for $\mathcal{D}_\lambda$, where $\lambda\in H^2(X^{reg},\C)$ corresponds to $c\in \param$). One can show that for  a suitable choice of $\mathcal{P}$, we have $\operatorname{End}_{\mathcal{D}_c^{opp}}(\mathcal{P}_c)\xrightarrow{\sim} H_c$, see
Proposition \ref{Prop:quant_Procesi_endom} below.

The following result implies Theorem \ref{Thm:der_equiv1}.

\begin{Thm}\label{Thm:der_equiv2}
The following is true:
\begin{enumerate}
\item When $c,c'$ have integral difference, the categories $\operatorname{Coh}(\mathcal{D}_c),
\operatorname{Coh}(\mathcal{D}_{c'})$ of coherent $\mathcal{D}_c$- and $\mathcal{D}_{c'}$-modules are
equivalent.
\item The functor $R\Gamma(\mathcal{P}_c\otimes_{\mathcal{D}_c}\bullet):D^b(\operatorname{Coh}(\mathcal{D}_c))
\xrightarrow{\sim} D^b(H_c\operatorname{-mod})$ is a derived equivalence.
\end{enumerate}
\end{Thm}

Let us describe some applications and consequences of our constructions.

Every pair $(X,\Pro)$ of a $\Q$-factorial terminalization $X$ of $V/\Gamma$ and a Procesi sheaf $\Pro$
on $X$ gives rise to a $t$-structure on $D^b(H_c\operatorname{-mod})$. In Section \ref{SS_Perv}, we will explain that
\cite[Theorem 3.1]{Perv} generalizes to our situation (and to a more general one):
some of the t-structures we consider are perverse to each other.

We will also establish the following result ((2) was conjectured by Etingof and Ginzburg in
\cite{EG2}):

\begin{Thm}\label{Thm:B_ineq}
The following is true for all $c\in \param$.
\begin{enumerate}
\item The regular $H_c$-bimodule $H_c$ has finite length.
\item Generalized Bernstein inequality holds for $H_c$, i.e., $$\operatorname{GK-}\dim(M)\geqslant
\frac{1}{2}\operatorname{GK-}\dim \left(H_c/\operatorname{Ann}(M)\right)$$
for any finitely generated $H_c$-module $M$.
\end{enumerate}
\end{Thm}

Reduction of (2) to (1) was done in \cite[Theorem 1.1]{B_ineq}. We will explain necessary modifications of arguments
from \cite{B_ineq} that prove (1).

{\bf Acknowledgements}. I would like to thank Roman Bezrukavnikov, Dmitry Kaledin, Emanuele Macri,
Yoshinori Namikawa and Ben Webster for stimulating discussions. I also would like to
thank the referees for many useful comments that allowed me to improve the exposition.
This work has been funded by HSE University Basic Research Program
by the Russian Academic Excellence Project '5-100'.
This work was also partially supported by the NSF under grant DMS-1501558.

\section{Preliminaries}
\subsection{Symplectic reflection algebras}\label{SS_SRA}
Let us start by recalling the definition of a symplectic reflection algebra $H_c$ and
its spherical subalgebra, due to Etingof and Ginzburg, \cite{EG}.

Let, as before, $V$ be a finite dimensional
vector space with symplectic form $\omega$ and $\Gamma\subset \operatorname{Sp}(V)$
be a finite subgroup. Let $S\subset \Gamma$ denote the set of symplectic reflections in
$\Gamma$, it is the union of $\Gamma$-conjugacy classes. Let $\param$ denote the space
of $\Gamma$-invariant maps $S\rightarrow \C$. For $s\in S$, let $\omega_s$ denote the skew-symmetric
form whose kernel is $\operatorname{ker}(s-1)$ and whose restriction to $\operatorname{im}(s-1)$
coincides with that of $\omega$. For $c\in S$ we set
$$H_c=T(V)\#\Gamma/(u\otimes v-v\otimes u-\omega(u,v)-\sum_{s\in S}c_s\omega_s(u,v)s).$$
This algebra comes with a filtration inherited from $T(V)\#\Gamma$.
The PBW property established in \cite[Theorem 1.3]{EG} says that $\gr H_c=\C[V]\#\Gamma$.

Also we can consider the universal version $H_{\param}$, a $\C[\param]$-algebra
that specializes to $H_c$ for $c\in \param$. Its associated graded algebra
$\gr H_{\param}$ is a flat graded deformation of $\C[V]\#\Gamma$ over $\C[\param]$.

Now take the averaging idempotent $e\in \C\Gamma, e:=\frac{1}{|\Gamma|}\sum_{\gamma\in \Gamma}\gamma$.
We can consider the subalgebra
$eH_ce\subset H_c$ with unit $e$. This is a filtered algebra with $\gr eH_ce=\C[V]^\Gamma$.

\subsection{$\Q$-factorial terminalizations}\label{SS_Q_fac_term}
The variety $Y:=V/\Gamma$ is a {\it conical symplectic singularity} meaning that $Y$
\begin{itemize}
\item
has symplectic singularities in the sense of Beauville, \cite[Definition 1.1]{Beauville},
\item and comes equipped with an action of a torus $\C^\times$ contracting $Y$ to  a single point and rescaling the Poisson structure on $Y$ via the
character $t\mapsto t^{-d}$ for a positive integer $d$. \end{itemize}
Indeed, the dilation action
of $\C^\times$ on $V$ descends to $V/\Gamma$ and $d=2$.

Now let $Y$ be an arbitrary conical symplectic singularity. The following result is well-known
but we were unable to find the proof in the literature, so we prove it.

\begin{Prop}\label{Prop:Q_fac_terminalization}
There is a variety $X$ and a projective birational morphism $\rho:X\rightarrow Y$
with the following properties:
\begin{enumerate}
\item $X$ is normal and $\Q$-factorial,
\item $X$ is Poisson and has symplectic singularities,
\item $\operatorname{codim}_{X}X^{sing}\geqslant 4$,
\item $\rho$ is a Poisson (in particular, crepant) morphism,
\item There is a $\C^\times$-action on
$X$ making $\rho$ equivariant.
\end{enumerate}
\end{Prop}

Note that by \cite{Namikawa_symplectic}, modulo (2), condition (3) is equivalent to $X$ being terminal.
So $X$ is called a $\Q$-factorial terminalization of $Y$. In general, it is not unique,
we will elaborate on that below.

\begin{proof}
The variety $Y$ has symplectic and hence canonical singularities. So $Y$ with the zero divisor
satisfies the conditions of \cite[Corollary 1.4.3]{BCHM} (note that our $Y$ is $X$ in
that corollary and our $X$ is their $Y$). Take a resolution of singularities $f: W\rightarrow Y$.
Let $\mathfrak{C}$ denote the set of all components of the exceptional divisor of $f$
with  discrepancy zero. By \cite[Corollary 1.4.3]{BCHM}, there is a normal $\Q$-factorial
variety $X$ with a projective morphism $\rho:X\rightarrow Y$ so that $\mathfrak{C}$ is precisely
the set of components of the exceptional divisor of $\rho$. In particular, $\rho$ is crepant.
Since $Y$ has symplectic singularities, the symplectic form on $Y^{reg}$ lifts to
$X^{reg}$. Since $\rho$ is crepant, that extension is symplectic. We conclude that
$X$ has symplectic singularities.

Now we prove (3). Consider the rational
map $W\dashrightarrow Y$ that intertwines $f$ and $\rho$. We can find a smooth
$\tilde{W}$ with projective birational morphisms  $f':\tilde{W}\rightarrow W$
and $\rho': \tilde{W}\rightarrow Y$ intertwined by $W\dashrightarrow Y$.
In particular, $f\circ f'=\rho\circ \rho'$
Every component of the exceptional divisor of $f'$ has positive  discrepancy
because $W$ is smooth. So $\mathfrak{C}$ is  the set of all components of
the exceptional divisor of $\rho\circ\rho'$ that have discrepancy $0$.
So every component of $\rho': \tilde{W}\rightarrow X$ has positive  discrepancy.

Now suppose that $\operatorname{codim}_X X^{sing}< 4$. Since $X$ has symplectic
singularities, it has finitely many symplectic leaves, \cite{Kaledin}.
So $X$ has a codimension $2$ symplectic leaf. Over a formal slice to a point in
this leaf, $\rho'$ factors through a minimal resolution. So $\rho'$ must have
a component of discrepancy $0$ in its exceptional divisor, a contradiction.
The proof of (1)-(4) is now complete.

(5) follows from \cite[Proposition A.7]{Namikawa_flop}.
\end{proof}


The variety $X$ has the following properties.

\begin{Lem}\label{Lem:X_properties}
The following is true:
\begin{enumerate}
\item $X$ is a Cohen-Macaulay and Gorenstein
variety with trivial canonical sheaf.
\item $\C[X]=\C[Y]$ and $H^i(X,\mathcal{O}_X)=0$ for $i>0$.
\end{enumerate}
\end{Lem}
\begin{proof}
(1) follows from (2) of Proposition \ref{Prop:Q_fac_terminalization}.
Now (2) follows from (1) of this lemma and the Grauert-Riemenschneider
theorem.
\end{proof}

We will also need the following property of $\rho$ established, for example, in the proof
of \cite[Proposition 2.14]{orbit}.

\begin{Lem}\label{Lem:ssmall}
The morphism $\rho$ is semismall.
\end{Lem}

Now let us discuss fields of definition. Suppose that $Y$ is a conical symplectic singularity
defined over $\overline{\mathbb{Q}}$, the quotient $Y=V/\Gamma$ provides an example.
Let us write $Y^{sr}$ for the union of the open
and all codimension two leaves in $Y$. Set $X^{sr}:=X\times_Y Y^{sr}$. This is an open
subvariety of $X$ whose complement has codimension at least 2 by Lemma \ref{Lem:ssmall}.
The restriction of $\rho$ to $X^{sr}$ is the minimal resolution of singularities. It follows
that $X^{sr}$ and $\rho|_{X^{sr}}$ are defined over $\overline{\mathbb{Q}}$.
Since $X$ is normal and the codimension of the complement to $X^{sr}$ is at least two,
it follows that $(X,\rho)$ are defined over $\overline{\Q}$.

It follows that $\rho:X\rightarrow Y$ is defined over some finite integral extension $R$ of a finite
localization of $\Z$. So for each prime $p$ which is large enough we can consider
a reduction $\rho_{\F}:X_{\F}\rightarrow Y_{\F}$ mod $p$, where $\F$ is an algebraically
closed field of characteristic $p$. The triple $(X_{\F},Y_{\F},\rho_{\F})$ is defined over
a finite field $\F_q$ for all $q=p^m$ with $m$ sufficiently large. Lemmas \ref{Lem:X_properties},\ref{Lem:ssmall} still hold.

Let us now recall the Namikawa-Weyl group, \cite{Namikawa2}. Let $Y$ be a conical symplectic singularity.
The Namikawa-Weyl group is defined as follows. Let $L_1,\ldots,L_k$ be the codimension
$2$ symplectic leaves of $Y$. The formal slice to each $L_i$ is a Kleinian singularity
and so gives rise to the Cartan space $\tilde{\mathfrak{h}}_i$ and the Weyl group $\tilde{W}_i$
of the corresponding ADE type. The fundamental group $\pi_1(L_i)$ acts on $\tilde{\mathfrak{h}}_i,
\tilde{W}_i$ by monodromy. Set $\mathfrak{h}_i=\tilde{\mathfrak{h}}_i^{\pi_1(L_i)},
W_i=\tilde{W}_i^{\pi_1(L_i)}$. By the Namikawa-Weyl group $W_Y$ we mean the product
$\prod_{i=1}^k W_i$, it acts  on $\h:=H^2(X^{reg},\C)=H^2(Y^{reg},\C)\oplus \bigoplus_{i=1}^\ell
\mathfrak{h}_i$ as a crystallographic reflection group.

Now let us discuss Poisson deformations of $X,Y$, see \cite{Namikawa1},\cite{Namikawa2},\cite{Namikawa_bir_geom}.
First, there is a universal Poisson deformation
$X_\h$ of $X$ over $\h$, it comes with a contracting $\C^\times$-action that restricts
to the contracting $\C^\times$-action on $X$ and induces a scaling action on $\h$.
The affinization  $Y_\h:=\operatorname{Spec}(\C[X_\h])$ of $X_\h$ is a deformation of $Y$ over $\h$.
Let us write $X_\lambda,Y_\lambda$ for the fibers of  $X_\h,Y_\h$ over $\lambda\in \h$.
We have $X_\lambda\twoheadrightarrow Y_\lambda$. The locus of $\lambda\in \h$, where $X_\lambda\rightarrow Y_\lambda$
is not an isomorphism, can be shown to be a union of hyperplanes, we denote it by $\h^{sing}$.

Using the Namikawa-Weyl group and the hyperplane arrangement $\h^{sing}$ we can also
classify all $\Q$-factorial terminalizations of  $Y$.
Consider the subspace $\h_{\R}$ of real points in $\h$. The hyperplane arrangement
$\h^{sing}$ splits $\h_{\R}$ into a $W_Y$-stable union of cones to be called
{\it chambers}. The following
is the main result of \cite{Namikawa_bir_geom}.

\begin{Lem}\label{Lem:Q_fac_classification}
The movable cone for $Y$ is a fundamental chamber for $W_Y$. Moreover, the $\Q$-factorial
terminalizations are classified by the chambers inside the movable cone, where we
send a $\Q$-factorial terminalization to its ample cone.
\end{Lem}

Below we will need to understand the formal deformations of formal neighborhoods $X^{\wedge_x}$
of $x\in X$ coming from $X_\h$. The idea of the proof of the next lemma was explained to
me by Kaledin. Let us write $\h^{\wedge_0}$ for the formal neighborhood of $0$ in $\h$.

\begin{Lem}\label{Lem:form_deform_induced}
Let $X$ be as above.
Let $X^{\wedge_x}$ and $X_\h^{\wedge_x}$ denote the formal neighborhoods of $x$ in
$X$ and $X_\h$, respectively. Then we have an isomorphism of formal Poisson
$\h^{\wedge_0}$-schemes $X_\h^{\wedge_x}\cong X^{\wedge_x}\times \h^{\wedge_0}$.
\end{Lem}
\begin{proof}
The formal Poisson deformations of $X^{\wedge_x}$ are classified by $H^2_{DR}(\left(X^{\wedge_x}\right)^{reg})$,
this is proved similarly to the main result of \cite{Namikawa1}, see the introduction there.
So we need to show that the pull-back map $H^2_{DR}(X^{reg})\rightarrow
H^2_{DR}(\left(X^{\wedge_x}\right)^{reg})$ is zero. Note that the natural map
$\C\otimes_{\Z}\operatorname{Pic}(X)\rightarrow\C\otimes_{\Z}\operatorname{Pic}(X^{reg})$
is an isomorphism because $X$ is $\Q$-factorial. Also note that the Chern character map induces an
isomorphism $\C\otimes_{\Z}\operatorname{Pic}(X^{reg})\xrightarrow{\sim} H^2_{DR}(X^{reg})$,
compare to the beginning of \cite[Section 2.3]{BPW}.
Clearly, the restriction of a line bundle on $X$ to $X^{\wedge_x}$ is trivial. It follows that
the map $H^2_{DR}(X^{reg})\rightarrow H^2_{DR}(\left(X^{\wedge_x}\right)^{reg})$
is zero and finishes the proof of the lemma.
\end{proof}


As Namikawa proved, \cite[\S 1]{Namikawa2}, the action of $W$ on $\h$ lifts to $Y_\h$ and $Y_\h/W$ is a universal
$\C^\times$-equivariant Poisson deformation of $Y$.

Now let us discuss the Picard group of $X^{reg}$. These groups are naturally identified
for different choices of $X$, compare to \cite[Proposition 2.18]{BPW}.
As was mentioned in the proof of Lemma \ref{Lem:form_deform_induced}, the Chern character map
$\operatorname{Pic}(X^{reg})\otimes_{\Z}\C\rightarrow \h=H^2(X^{reg},\C)$
is an isomorphism. Let $\h_{\Z}\subset \h$ denote the image
of $\operatorname{Pic}(X^{reg})$.  In the case when $H^2(Y^{reg},\C)=0$ (this holds,
for instance, for $Y=V/\Gamma$, see, e.g., \cite[Lemma 2.4]{Bellamy})
we can describe $\h_{\Z}$ as follows. In each $\h_i$ we have the coweight
lattice and the lattice $\h_{\Z}$ is the direct sum of those.

In our main example, $Y=V/\Gamma$, the codimension 2 symplectic leaves are in bijection with
$\Gamma$-conjugacy classes of subgroups $\underline{\Gamma}\subset \Gamma$ such that $\dim V^{\underline{\Gamma}}=
\dim V-2$. The Kleinian group corresponding to this leaf is $\underline{\Gamma}$ and the fundamental
group of the leaf is $N_{\Gamma}(\underline{\Gamma})/\underline{\Gamma}$. One can show, see, e.g., \cite[Section 3.2, Proposition 2.6]{Bellamy} or
\cite[Section 3.7]{orbit} that $\h$ is naturally identified with $\param$ as a vector space. Namely,
 for $i>0$, the space $\tilde{\h}_i$ is identified with
with the space $\tilde{\param}_i$ of $\Gamma_i$-invariant functions on $\Gamma_i\setminus \{1\}$.
So the space $\param_i$ of $N_\Gamma(\Gamma_i)$-invariant functions is identified with $\h_i$.
It was checked by Bellamy, \cite[Theorem 1.4]{Bellamy}, that the  deformation $Y_\h$ is
$\operatorname{Spec}(e \gr H_\param e)$.

\subsection{Quantizations of $\Q$-factorial terminalizations}\label{SS_quant}
Let again $Y$ be a general conical symplectic singularity. Let us discuss the quantizations of
$X$ and $Y$ following \cite{BPW,orbit}. Let us start with $X$. By the conical topology on $X$
we mean the topology where ``open'' means Zariski open and $\C^\times$-stable. By a quantization
of $X$ we mean a sheaf $\mathcal{D}$ of algebras in the conical topology on $X$ coming with
\begin{itemize}
\item a complete and separated, ascending exhaustive $\Z$-filtration,
\item a graded Poisson algebra isomorphism $\gr\mathcal{D}\xrightarrow{\sim}\mathcal{O}_X$.
\end{itemize}
It was shown in \cite[Sections 3.1,3.2]{BPW} that the filtered quantizations of $X$ are in a canonical bijection
with   $\mathfrak{h}$. Let us denote the quantization corresponding to $\lambda
\in \mathfrak{h}$ by $\mathcal{D}_\lambda$. We also have the universal quantization, $\mathcal{D}_\mathfrak{h}$,
a sheaf (in the conical topology) of filtered algebras on $X$ with the following properties:
\begin{itemize} \item its specialization to $\lambda$ coincides with $\mathcal{D}_\lambda$,
\item it quantizes the universal deformation $X_{\h}$ of $X$ over $\h$.
\end{itemize}
We also note, \cite[Section 2.3]{quant_iso}, that
\begin{equation}\label{eq:iso_oppos}
\mathcal{D}_{-\lambda}\cong \mathcal{D}_\lambda^{opp}.
\end{equation}

We have a quantum analog of Lemma \ref{Lem:form_deform_induced}. Namely, consider the
sheaf $\D_{\lambda\hbar^d}$ of $\C[[\hbar]]$-algebras on $X$ that is the $\hbar$-adic completion
of the Rees sheaf $R_\hbar(\D_\lambda)$. Similarly, we can consider the sheaf $\D_{\h,\hbar}$
of $\C[[\h,\hbar]]$-algebras. The relation between these two sheaves is that $\D_{\lambda\hbar^d}$
is the specialization of $\D_{\h,\hbar}$ to $\lambda\hbar^d$. Now we can consider the restrictions
$\D_{\lambda\hbar^d}^{\wedge_x},\D_{\h,\hbar}^{\wedge_x}$.

\begin{Lem}\label{Rem:formal_quant}
We have $\D_{\h,\hbar}^{\wedge_x}\cong \C[[\h]]\widehat{\otimes}_{\C}\D_{\lambda\hbar^d}^{\wedge_x}$.
\end{Lem}

The proof is similar to that of   Lemma
\ref{Lem:form_deform_induced}

Let us proceed to quantizations of $Y$ (or, equivalently, of the filtered algebra $\C[Y]=\C[X]$).
Since $H^i(X,\mathcal{O}_X)=0$ for $i>0$ by Lemma \ref{Lem:X_properties},
we see that $\A_\lambda=\Gamma(\mathcal{D}_\lambda)$
is a quantization of $\C[Y]$ that coincides with the specialization of $\A_{\h}$ to $\lambda$.
It was shown in \cite[Section 3.3]{BPW} that the group $W=W_Y$ acts on $\A_{\h}$ by filtered algebra
automorphisms so that the induced action on $\h$ coincides with the original one.
Moreover, it was shown in \cite[Proposition 3.5]{orbit} that, under a mild additional assumption,
$\A_{\h}^W$ is a universal filtered quantization of $\C[Y]$, and, in particular, every filtered
quantization of $\C[Y]$ has the form $\A_\lambda$ for some $\lambda\in \h$.
The assumption is that $\C[Y]_i=0$ for $0<i<d$, it holds for $Y:=V/\Gamma$ when
$V^{\Gamma}=\{0\}$.
The universality is understood in the following sense. Let $B$ be a
finitely generated graded commutative $\C$-algebra.
Let $\A_B$ be a filtered $B$-algebra (so that $B\rightarrow A_B$ is a filtered
algebra homomorphism) such that $\gr \A_B$ is a graded Poisson deformation of $\C[Y]$.
Then there is a unique filtered algebra homomorphism $\C[\h]^W\rightarrow B$ and a unique
filtered $B$-algebra isomorphism $B\otimes_{\C[\h]^W}\A_\h^W\rightarrow \A_B$ such that
the associated graded homomorphism gives the identity automorphism of $\C[Y]$. The assumption
$\C[Y]_i=0$ for $0<i<d$ holds for $Y=V/\Gamma$ assuming $V^\Gamma=\{0\}$.

Let us compare the quantizations $\A_\lambda$ and $e H_c e$ of $\C[V]^\Gamma$,
\cite[Section 3.7]{orbit}. There is an affine identification $\lambda\mapsto c(\lambda): \mathfrak{h}
\rightarrow \param$ (whose linear part was mentioned in the end of the previous section) such that
$eH_{\param}e\cong \A_\h$ (an isomorphism of filtered algebras compatible with the
isomorphism $\lambda\mapsto c(\lambda)$). In particular, $\A_\lambda\cong eH_{c(\lambda)}e$
(an isomorphism of filtered quantizations of $\C[V]^\Gamma$).

\subsection{Harish-Chandra bimodules over quantizations of $\C[Y]$}\label{SS_HC_quant}
Let us  discuss  Harish-Chandra (HC) bimodules over quantizations of a conical
symplectic singularity  $Y$. Let $\gamma\in \h$.
By definition, a HC $(\A_{\h},\gamma)$-bimodule  is a bimodule $\B$ that can be equipped with a
bounded below filtration $\B=\bigcup_i \B_{\leqslant i}$ (to be called a good filtration)
such that
\begin{itemize}
\item[(i)]  $[a,b]=\langle\gamma,a\rangle b$ for $a\in \h^*\subset \A_\h$,
\item[(ii)] $[(\A_{\h})_{\leqslant i}, \B_{\leqslant j}]\subset \B_{\leqslant i+j-d}$,
\item[(iii)] $\gr\B$ is a finitely generated $\C[Y_\h]$-module.
\end{itemize}
By a HC $\A_\h$-bimodule we mean a HC $(\A_\h,0)$-bimodule. The category
of HC $(\A_\h,\gamma)$-bimodules will be denoted by $\HC(\A_\h,\gamma)$
(and simply $\HC(\A_\h)$ when $\gamma=0$).

By definition, a HC $\A_{\lambda'}$-$\A_\lambda$-bimodule is a HC $(\A_\h,\lambda'-\lambda)$-bimodule
such that the action of $\h$ on the right factors through $\lambda$. We will write $\HC(\A_{\lambda'},\A_{\lambda})$
for the category of HC $\A_{\lambda'}$-$\A_\lambda$-bimodules.

Under some additional conditions on $Y$, we can define restriction functors between categories
of HC bimodules. Namely, assume that the formal slices to all symplectic leaves in $Y$
are conical, i.e., they come with contracting $\C^\times$-actions that rescale the Poisson
bracket by $t\mapsto t^{-d}$. By the proof of \cite[Proposition A.7]{Namikawa_flop},
we also have  actions on the pre-images of the slice in $X$. We can assume that $d$ is even
and is the same for all slices: otherwise we replace $d$ with a suitable multiple and
rescale the $\C^\times$-actions accordingly.

Pick a symplectic leaf $\mathcal{L}$ in $Y$. Let $\underline{Y}$
be the conical symplectic singularity such that its formal neighborhood at $0$ is
the formal slice to $\mathcal{L}$ at some point $y\in \mathcal{L}$. Arguing similarly
to \cite[Lemma 3.3]{Perv}, we see that the completion $R_\hbar(\A_\lambda)^{\wedge_x}$
of the Rees algebra of $\A_\lambda$ splits into the completed tensor product
$R_\hbar(\underline{\A}_\lambda)^{\wedge_0}\widehat{\otimes}_{\C[[\hbar]]}\mathbb{A}_\hbar(T_x\mathcal{L})^{\wedge_0}$.
Here we write $\mathbb{A}_\hbar(T_x\mathcal{L})$ for the homogenized Weyl algebra
of the symplectic vector space $T_x\mathcal{L}$:  $\mathbb{A}_\hbar(T_x\mathcal{L})=T(T_x\mathcal{L})[\hbar]/(u\otimes v-v\otimes u-\hbar^{d}\omega(u,v))$. Further, we write $\underline{\A}_\lambda$ for the filtered quantization
of $\C[\underline{Y}]$ that corresponds to the restriction of $\lambda\in H^2(X^{reg},\C)$
to $\underline{X}^{reg}$ (where $\underline{X}$ is the $\Q$-factorial terminalization of $\underline{Y}$
that is characterized by the property that $\underline{Y}^{\wedge_0}\times_{\underline{Y}} \underline{X}$
is the preimage of the formal slice to $\mathcal{L}$ in $X$, compare with \cite[Section 3.2]{Perv}).

Note that $\C[Y]^{\wedge_x}$ comes equipped with two derivations, $D,\underline{D}$ that multiply
$\{\cdot,\cdot\}$ by $-d$. Arguing as in the proof of (1) of \cite[Lemma 3.3]{Perv} and
using \cite[Lemma 2.15]{orbit}, we show that all Poisson derivations of $\C[Y]^{\wedge_x}$
are inner. Now we can argue as in \cite[Section 3.3]{Perv}  to construct
the restriction functor $\bullet_{\dagger,x}:\HC(\A_{\lambda'}\operatorname{-}\A_{\lambda})\rightarrow
\HC(\underline{\A}_{\lambda'}\operatorname{-}\underline{\A}_\lambda)$. Direct analogs of
\cite[Lemma 3.5, Lemma 3.6, Lemma 3.7]{Perv} hold (in the case of $Y=V/\Gamma$ Lemmas 3.5 and 3.7 were
established in \cite[Section 3.6]{sraco}, while Lemma 3.6 easily follows from the construction) with the
same proofs. We recall these results here for reader's convenience.

For $\B\in \HC(\A_{\lambda'},\A_{\lambda})$, by the {\it associated variety} $\VA(\B)$ we mean the
support of the $\C[Y]$-module $\gr\B$ in $Y$, where the associated graded is taken with respect
to any good filtration on $\B$.

\begin{Lem}[Lemma 3.5 in \cite{Perv}]\label{Lem:dagger_assoc_var}
For every HC $\A_{\lambda'}$-$\A_\lambda$-bimodule, the associated variety $\VA(\B_{\dagger,x})$ is the unique
conical subvariety in $\underline{Y}$ such that
$\VA(\B_{\dagger,x})\times \mathcal{L}^{\wedge_x}=\VA(\B)^{\wedge_x}$.
\end{Lem}

\begin{Lem}[Lemma 3.6 in \cite{Perv}]\label{Lem:dagger_Tor_Ext}
The functor $\bullet_{\dagger,x}$ intertwines the Tor and the Ext functors. More precisely,
we get the following:
\begin{align*}
&\operatorname{Tor}^{\A_\lambda}_i(\B^1,\B^2)_{\dagger,x}\cong
\operatorname{Tor}^{\underline{\A}_\lambda}_i(\B^1_{\dagger,x}, \B^2_{\dagger,x}),\\
&\operatorname{Ext}_{\A_\lambda}^i(\B^1,\B^2)_{\dagger,x}\cong
\operatorname{Ext}_{\underline{\A}_\lambda}^i(\B^1_{\dagger,x}, \B^2_{\dagger,x}). \end{align*}
The latter equality holds for Ext's of left modules and of right modules.
\end{Lem}

Let us write $\HC_{\overline{\mathcal{L}}}(\A_{\lambda'},\A_{\lambda})$ for the full
subcategory of $\{\B\in \HC(\A_{\lambda'},\A_{\lambda})| \VA(\B)\subset \overline{\mathcal{L}}\}$. We write
$\HC_{fin}(\underline{\A}_{\lambda'},\underline{\A}_\lambda)$ for
the category of finite dimensional bimodules. By Lemma
\ref{Lem:dagger_assoc_var}, the functor
$\bullet_{\dagger}$ maps $\HC_{\overline{\mathcal{L}}}(\A_{\lambda'},\A_{\lambda})$
to $\HC_{fin}(\underline{\A}_{\lambda'},\underline{\A}_\lambda)$.

\begin{Lem}[Lemma 3.7 in \cite{Perv}]\label{Lem:adj_functor_old}
The functor $\bullet_{\dagger,x}:\HC_{\overline{\mathcal{L}}}(\A_{\lambda'},\A_{\lambda})\rightarrow
\HC_{fin}(\underline{\A}_{\lambda'},\underline{\A}_\lambda)$ admits a right adjoint functor,
to be denoted by $\bullet^{\dagger,x}$.
\end{Lem}

Finally, let us discuss supports of HC bimodules in $\h$. Let $\B\in \HC(\A_\h,\gamma)$. By the (right) support
$\operatorname{Supp}^r_\h(\B)$ of $\B$ in $\h$, we mean $\{\lambda\in \h| \B\otimes_{\C[\h]}\C_\lambda\neq 0\}$.
Similarly to \cite[Proposition 2.6]{O_resol}, we get the following result.

\begin{Prop}\label{Prop:support}
The subvariety $\operatorname{Supp}^r_\h(\B)\subset \h$ is closed and its asymptotic cone coincides with
$\operatorname{Supp}_\h(\gr\B)$, where the associated graded is taken with respect to any good filtration on
$\B$.\end{Prop}

\section{Simplicity of $\mathcal{D}_\lambda$}
\subsection{Main conjecture and result}
Let $Y$ be an arbitrary conical symplectic singularity and let $X$ be its $\Q$-factorial terminalization.
Let $\mathcal{D}$ be a quantization of $X$.

We conjecture the following.

\begin{Conj}\label{Conj:symplicity}
The sheaf of algebras $\mathcal{D}$ is simple.
\end{Conj}

In the case when $X$ is smooth, the conjecture is standard because $\mathcal{O}_X$ has no Poisson ideals.

The main result of this section is a proof of this conjecture in the case when $Y=V/\Gamma$
(Theorem \ref{Thm:simplicity} from the introduction).

The most essential result about $V/\Gamma$ that we use
is that the algebra $\A_\lambda$ is simple for  $\lambda$ outside a countable union of
algebraic subvarieties of $\h$, such parameters are called {\it Weil generic}.
This result follows from \cite[Theorem 4.2.1]{sraco}. This result is not known in the general case,
even when Y is Q-factorial and terminal, in which case there is only one quantization..

\subsection{Abelian localization}\label{SS_ab_loc}
Let $Y$ again be an arbitrary conical symplectic singularity.

Pick $\lambda\in \h$. We can consider the category $\operatorname{Coh}(\mathcal{D}_\lambda)$
of coherent $\mathcal{D}_\lambda$-modules. There is the global section functor
$\Gamma_\lambda: \operatorname{Coh}(\mathcal{D}_\lambda)\rightarrow \A_\lambda\operatorname{-mod}$
that has left adjoint $\operatorname{Loc}_\lambda:=\mathcal{D}_\lambda\otimes_{\A_\lambda}\bullet$.
We say that abelian localization holds for $(X,\lambda)$ if these functors are mutually inverse equivalences.

Note that we have an inclusion $\operatorname{Pic}(X)\hookrightarrow \operatorname{Pic}(X^{reg})$.
Since $X$ is $\Q$-factorial, the induced map $\operatorname{Pic}(X)\otimes_{\Z}\Q\rightarrow \operatorname{Pic}(X^{reg})\otimes_{\Z}\Q$
is an isomorphism. Pick $\chi\in \operatorname{Pic}(X)$ and let $\mathcal{O}_X(\chi)$
denote the corresponding line bundle on $X$.  Abusing the notation we will denote
its image in $\h$ also by $\chi$.

The following result is a direct generalization of results of \cite[Section 5.3]{BPW}.

\begin{Prop}\label{Prop:ab_loc_weak}
Suppose $\chi$ is ample for $X$. For any $\lambda\in \h$, there is $n_0\in \Z$ with the property that abelian localization holds for $(X,\lambda+n\chi)$ for any $n\geqslant n_0$.
\end{Prop}

In Section \ref{SS_ab_simpl} we will establish a stronger version of this result.

The scheme of proof of Proposition \ref{Prop:ab_loc_weak}
is similar to what is used in \cite{BPW} and is based on translation bimodules.

The line bundle $\mathcal{O}(\chi)$ on $X$ quantizes to a $\mathcal{D}_{\lambda+\chi}$-$\mathcal{D}_\lambda$-bimodule
(that is a sheaf on $X$) to be denoted by $\mathcal{D}_{\lambda,\chi}$. Note that tensoring with $\mathcal{D}_{\lambda,\chi}$ gives an equivalence $\operatorname{Coh}(\mathcal{D}_{\lambda})\xrightarrow{\sim}\operatorname{Coh}(\mathcal{D}_{\lambda+\chi})$,
compare to \cite[Section 5.1]{BPW}. Set $\A_{\lambda,\chi}:=\Gamma(\mathcal{D}_{\lambda,\chi})$, this is an $\A_{\lambda+\chi}$-$\A_\lambda$-bimodule. It is easy to see that it is HC, compare
to \cite[Proposition 6.24]{BPW}.

We can also consider the universal version of $\mathcal{D}_{\lambda,\chi}$, the $\mathcal{D}_{\h}$-bimodule
$\mathcal{D}_{\h,\chi}$, so that $\mathcal{D}_{\lambda,\chi}=\mathcal{D}_{\h,\chi}\otimes_{\C[\h]}\C_\lambda$.
Let $\A_{\h,\chi}:=\Gamma(\mathcal{D}_{\h,\chi})$, it is an object of $\operatorname{HC}(\A_\h,\chi)$.

The following lemma was established in \cite[Section 6.3]{BPW} in the case when $X$ is smooth.
Recall that, for different $\Q$-terminalizations $X,X'$, the groups $\operatorname{Pic}(X^{reg})$
and $\operatorname{Pic}(X'^{reg})$ are naturally identified.

\begin{Lem}\label{Lem:transl_properties}
The following claims are true.
\begin{enumerate}
\item Let $X,X'$ be two $\Q$-terminalizations of $Y$.
If $\chi\in \operatorname{Pic}(X^{reg})$ is contained in both
$\operatorname{Pic}(X),\operatorname{Pic}(X')$, then the bimodule
$\A_{\h,\chi}$ is the same for $X$ and $X'$.
\item Suppose $H^1(X,\mathcal{O}(\chi))=0$. Then $\A_{\lambda,\chi}=\A_{\h,\chi}\otimes_{\C[\h]}\C_\lambda$.
\end{enumerate}
\end{Lem}

We provide the proof for reader's convenience.

\begin{proof}
We start with (1).  Consider the morphisms $X_{\h}\rightarrow Y_\h, X'_\h\rightarrow Y_\h$.
Both are isomorphisms over $\h^{sing}$. The locus of all points $y\in Y_\h$ such that
$X_\h\rightarrow Y_\h$ is an  isomorphism over an open neighborhood of $y$ consists
of all points with a $\Q$-factorial terminal singularity in the fiber of $Y_\h\rightarrow \h$.
The  locus of all such $y$ is open,
it will be denoted by $Y^0_\h$. Note that $Y^0_\h\cup \h^{sing}\times_\h Y_\h=Y_\h$
and that the open leaf in $Y_\lambda$ lies in $Y^0_\h$ for every $\lambda\in \h$.
It follows that $\operatorname{codim}_{Y_\h}Y_\h\setminus Y_\h^0\geqslant 3$,
hence $H^1(Y_{\h}^0, \mathcal{O})=0$. Therefore, two quantizations of the same
line bundle on $Y_\h^0$ are isomorphic.

Now let $\A_{\h,\chi},\A'_{\h,\chi}$ be the bimodules coming from $X,X'$.
By the previous paragraph, the microlocaluztions of $\A_{\h,\chi},\A'_{\h,\chi}$
to $Y_\h^0$ are isomorphic. Also the analysis of the previous paragraph
implies that the codimension of  $X_\h\setminus Y_\h^0$ in both
$X_\h,X'_\h$ is at least $2$. So $\A_{\h,\chi}=\Gamma(\A_{\h,\chi}|_{Y_\h^0})=\A'_{\h,\chi}$.
This proves (1).

Now we prove (2). From $H^1(X,\mathcal{O}(\chi))=0$ it follows that $H^1(X, \mathcal{D}_{\h_1,\chi})=0$
for any affine subspace $\h_1\subset \h$ containing $\lambda$, where we write $\mathcal{D}_{\h_1,\chi}$
for $\C[\h_1]\otimes_{\C[\h]}\D_{\h,\chi}$. Using the standard long exact sequence in cohomology,
we conclude that $\Gamma(\D_{\h_2,\chi})=\C[\h_2]\otimes_{\C[\h_1]}\D_{\h_1,\chi}$ for any affine
subspace $\h_2\subset\h_1$ such that $\h_2$ has codimension $1$ in $\h_1$. This implies
(2).
\end{proof}


\begin{proof}[Proof of Proposition \ref{Prop:ab_loc_weak}]
We can assume that $H^i(X,\mathcal{O}(n\chi))=0$ for all $n,i>0$.
Similarly to \cite[Proposition 5.13]{BPW}, abelian localization holds for $(X,\lambda)$
provided each $\A_{\lambda+n\chi,\chi}$ is a Morita equivalence and for
all $m>0$ the natural map
$$\A_{\lambda+(m-1)\chi,\chi}\otimes_{\A_{\lambda+(m-1)\chi}}\A_{\lambda+(m-2)\chi,\chi}
\ldots\otimes_{\A_{\lambda+\chi}}\A_{\lambda,\chi}\rightarrow \A_{\lambda,m\chi}$$
is an isomorphism. Similarly to the proof of \cite[Lemma 4.4]{BL}, we see that
the latter will follow if we show that $\A_{\h,-\chi}|_{\lambda+(n+1)\chi}$ is inverse to
$\A_{\lambda+n\chi,\chi}$ for all $n\geqslant 0$.

Following \cite[Section 2.2.5]{BL}, we say that a Zariski open subset  $U\subset\h$ is {\it asymptotically generic},
if the asymptotic cone of $\h\setminus U$ is contained in $\h^{sing}$. Arguing as in the proof of
\cite[Proposition 4.5(2)]{BL}, we see that the locus, where $\A_{\lambda,\chi},\A_{\h,-\chi}|_{\lambda+\chi}$
are mutually inverse Morita equivalences is asymptotically generic. In particular, its intersection with
the line $\{\lambda+z\chi| z\in \C\}$ is nonempty. This finishes the proof.
\end{proof}

\begin{Cor}\label{Cor:Weil_gen_simpl}
Assume that the algebra $\A_\lambda$ is simple for a Weil generic $\lambda$.
Then, for a Weil generic $\lambda\in \h$, the sheaf of algebras $\mathcal{D}_\lambda$ is simple.
\end{Cor}
\begin{proof}
Recall, (\ref{eq:iso_oppos}), that $\mathcal{D}_\lambda^{opp}\cong \mathcal{D}_{-\lambda}$.
We can view $\mathcal{D}_\lambda\widehat{\otimes} \mathcal{D}_{-\lambda}$ as a quantization of $X\times X$.

Let us show that $X\times X$ is $\Q$-factorial and terminal. It is normal and it is easy to see that
it has symplectic singularities. Recall that $\operatorname{codim}_X X^{sing}\geqslant 4$.
From here we deduce $\operatorname{codim}_{X\times X}(X\times X)^{sing}\geqslant 4$.
So $X\times X$ is terminal. To prove that $X\times X$ is $\Q$-factorial
we observe that $\operatorname{Cl}(X\times X)=\operatorname{Cl}(X)\oplus \operatorname{Cl}(X)$.
Since $X$ is $\Q$-factorial, then so is $X\times X$. In particular, $X\times X$ is a $\Q$-factorial
terminalization of $Y\times Y$.

The proof of
Proposition \ref{Prop:ab_loc_weak} shows that abelian localization
holds for $(X\times X,(\lambda,-\lambda))$ assuming $\lambda$ is in the intersection of
integral translates of some asymptotically generic Zariski open subset, in particular, when
$\lambda$ is Weil generic. Note that the global section functor sends the regular
bimodule  $\mathcal{D}_\lambda$ to the regular bimodule $\A_\lambda$. The latter is simple,
so is the former.
\end{proof}

\subsection{Leaves in $X$ and two-sided ideals in $\mathcal{D}_\lambda$}

In this section, we will prove technical results that are
analogous to several results obtained in \cite[Section 3]{B_ineq} and will be used
in the proof of Theorem \ref{Thm:simplicity}.

So let $Y$ be a conical symplectic singularity, $X$ be its $\Q$-terminalization,
$\h:=H^2(X^{reg},\C)$.  The variety $X$ has finitely many symplectic leaves, \cite{Kaledin}.
Let $\mathcal{L}$ be one of these leaves.
Note that the $\C^\times$-action preserves $\mathcal{L}$ and the action on the closure of $\mathcal{L}$ is contracting.

Pick a point $x\in \mathcal{L}$. Consider the formal neighborhood $\mathcal{L}^{\wedge_x}$ and its algebra of functions
$\C[\mathcal{L}^{\wedge_x}]$. This is a Poisson algebra. The action of $\C^\times$ on $X$ induces
a derivation of $\C[\mathcal{L}^{\wedge_x}]$ that rescales the Poisson bracket. We call it the Euler derivation
and denote it by $\mathsf{eu}$. Consider the category $\mathcal{C}(\mathcal{L}^{\wedge_x})$ of all finitely
generated Poisson $\C[\mathcal{L}^{\wedge_x}]$-modules that come equipped with an Euler derivation,
compare with \cite[Section 3.2]{B_ineq}. On the other hand, consider the category $\mathcal{C}(\mathcal{L})$
consisting of all finitely generated weakly $\C^\times$-equivariant Poisson $\mathcal{O}_{\mathcal{L}}$-modules.
We have the functor $\bullet_{\dagger,x}: \mathcal{C}(\mathcal{L})\rightarrow \mathcal{C}(\mathcal{L}^{\wedge_x})$
of completing at $x$.

\begin{Lem}\label{Lem:adj_functor}
Assume that the algebraic fundamental group of $\mathcal{L}$ is finite.
The functor $\bullet_{\dagger,x}: \mathcal{C}(\mathcal{L})\rightarrow \mathcal{C}(\mathcal{L}^{\wedge_x})$
admits a right adjoint functor.
\end{Lem}
\begin{proof}
Let $\pi:\tilde{\mathcal{L}}\rightarrow \mathcal{L}$ denote the universal algebraic cover that exists
because the algebraic fundamental group is finite. The action of $\C^\times$ lifts to
$\tilde{\mathcal{L}}$ possibly after replacing the given $\C^\times$ with a covering $\C^\times$.
So we can consider the category $\mathcal{C}(\tilde{\mathcal{L}})$ that comes
with adjoint functors $\pi^*: \mathcal{C}(\mathcal{L})\rightarrow \mathcal{C}(\tilde{\mathcal{L}}),
\pi'_*:\mathcal{C}(\tilde{\mathcal{L}})\rightarrow \mathcal{C}(\mathcal{L})$, where we write
$\pi'_*$ for the equivariant descent. We still have the functor $\bullet_{\tilde{\dagger},x}: \mathcal{C}(\tilde{\mathcal{L}})\rightarrow \mathcal{C}(\mathcal{L}^{\wedge_x})$
that satisfies $\bullet_{\tilde{\dagger},x}\circ \pi^*\cong \bullet_{\dagger,x}$. So it is enough
to show that $\bullet_{\tilde{\dagger},x}$ admits a right adjoint functor, say $\bullet^{\tilde{\dagger},x}$.
Then the  right adjoint to $\bullet_{\dagger,x}$ is given by $\pi'_*\circ \bullet^{\tilde{\dagger},x}$.

Arguing as in Step 3 of the proof of \cite[Lemma 3.9]{B_ineq}, we see that every object
$M\in\mathcal{C}(\tilde{\mathcal{L}})$ is of the form $\mathcal{O}_{\tilde{\mathcal{L}}}\otimes V$,
where $V$ is a finite dimensional rational representation of
$\C^\times$. Similarly, any object  $N\in\mathcal{C}(\mathcal{L}^{\wedge_x})$ is of the form $\C[\mathcal{L}^{\wedge_x}]\otimes V'$, where $V'$ is a finite dimensional vector space with a linear operator
($V'$ arises as the Poisson center of $N$, the linear operator is obtained by restricting the Euler derivation).
The functor $\bullet_{\tilde{\dagger},x}$ becomes $V\mapsto V'$ (where we take the linear operator coming by differentiating the $\C^\times$-action). This functor clearly has right adjoint (that sends $V'$ to the sum of all eigenspaces
with integral eigenvalues).
\end{proof}

Let us give corollaries of this lemma. The first concerns two-sided ideals in quantizations.
Note that the formal quantization $\D_{\h,\hbar}^{\wedge_x}$ of $X_\h^{\wedge_x}$ again comes
equipped with an Euler derivation $\mathsf{eu}$ (that rescales $\hbar$). Let
$\mathcal{I}_\hbar\subset \mathcal{D}_{\h,\hbar}^{\wedge_x}$ be an $\mathsf{eu}$-stable two-sided ideal such that
$\mathcal{D}_{\h,\hbar}^{\wedge_x}/\mathcal{I}_\hbar$ is finitely generated over $\C[[\h]]\widehat{\otimes} \Weyl^{\wedge_0}_\hbar(T_x\mathcal{L})$ (here the second factor is the completed homogenized
Weyl algebra of the symplectic vector space $T_x\mathcal{L}$).


\begin{Prop}\label{Prop:max_ideal}
Assume that the algebraic fundamental group of $\mathcal{L}$ is finite.
Then there is the largest (with respect to inclusion) sheaf of ideals
$\mathcal{J}_\hbar\subset \mathcal{D}_{\h,\hbar}$ whose completion
at $x$ is contained in $\mathcal{I}_{\hbar}$. This ideal has the following properties:
\begin{enumerate}
\item It is $\C^\times$-stable.
\item The intersection of the support of $\mathcal{D}_{\h,\hbar}/\mathcal{J}_\hbar$ with $X$ (viewed as a subvariety of $X_\h$) is $\overline{\mathcal{L}}$.
\end{enumerate}
\end{Prop}
Note that \cite[Proposition 3.8]{B_ineq} is an affine analog of this proposition.
The proof of Proposition \ref{Prop:max_ideal} is similar to that and we will provide
it for reader's convenience.
\begin{proof}
Consider the category $\mathfrak{PB}$ of coherent Poisson $\D_{\h,\hbar}$-bimodules $\B$
that carry a $\C^\times$-action compatible
with that on $\D_{\h,\hbar}$ and satisfy the following two conditions:
\begin{enumerate}
\item the left and right actions of $\C[\h][[\hbar]]$ on $\mathcal{B}$
coincide.
\item $\B/(\h,\hbar)\B$ is a coherent sheaf on $X$.
\end{enumerate}
Inside we can consider the full subcategory $\mathfrak{PB}_{\overline{\mathcal{L}}}$
of all objects supported on $\overline{\mathcal{L}}$.

Also consider the category of $\mathfrak{PB}^{\wedge}$ of finitely
generated Poisson $\D_{\h,\hbar}^{\wedge_x}$-bimodules $\B'$ that come with an Euler derivation
and   satisfy the following three conditions:
\begin{enumerate}
\item the left and right actions of $\C[[\h,\hbar]]$ on $\B'$ coincide,
\item and $\B'/(\h,\hbar)\B'$ is a finitely generated $\C[X^{\wedge_x}]$-module.
\end{enumerate}
Similarly to the above we can define the full subcategory $\mathfrak{PB}^{\wedge}_{\mathcal{L}}
\subset \mathfrak{PB}^\wedge$.

We have the completion functor $\mathfrak{PB}\rightarrow \mathfrak{PB}^{\wedge}$.
It has the right adjoint on $\mathfrak{PB}^{\wedge}_{\mathcal{L}}$ that maps
to $\mathfrak{PB}_{\overline{\mathcal{L}}}$: we view $\B'\in \mathfrak{PB}^{\wedge}$ as a sheaf
on $X$ via push-forward from $X^{\wedge_x}$ and take the sum of all subbimodules
in $\B'$ lying in $\mathfrak{PB}$. That the sum is coherent follows
from Lemma \ref{Lem:adj_functor}. Let us write $\bullet_{fin}$ for this right
adjoint functor.

The ideal $\mathcal{J}_\hbar$ is the kernel of the natural map $\D_{\h,\hbar}
\rightarrow (\D_{\h,\hbar}^{\wedge_x}/\mathcal{I}_\hbar)_{fin}$.
\end{proof}

Another consequence of the proof of Proposition \ref{Prop:max_ideal} is the following claim.

\begin{Lem}\label{Lem:leaf_intersection}
Let $\Leaf$ be a symplectic leaf in $X$ with finite algebraic fundamental group. Let $\Leaf_1$ be a leaf in some
fiber of $X_\h\rightarrow \h$ such that $\overline{\Leaf}$ is an irreducible component in $\overline{\C^\times \mathcal{L}_1}\cap X$. Then $\overline{\Leaf}=X\cap \overline{\C^\times \mathcal{L}_1}$.
\end{Lem}
\begin{proof}
Let $\mathcal{I}_\hbar$ denote the kernel of $\D_{\h,\hbar}^{\wedge_x}\twoheadrightarrow \C[\overline{\C^\times \mathcal{L}_1}]^{\wedge_x}$. The ideal $\J_\hbar$ from Proposition \ref{Prop:max_ideal}
is the kernel of $\D_{\h,\hbar}\twoheadrightarrow \mathcal{O}_{\overline{\C^\times \mathcal{L}}}$.
Our claim follows from Proposition \ref{Prop:max_ideal}.
\end{proof}

\begin{Prop}\label{Lem:leave_fund_group}
The algebraic fundamental group of every leaf $\mathcal{L}\subset X$ is finite.
\end{Prop}
\begin{proof}
Consider a point $x\in \mathcal{L}$ and its formal neighborhood $X_{\h}^{\wedge_x}$. By Lemma \ref{Lem:form_deform_induced}, $X_\h^{\wedge_x}\cong X^{\wedge_x}\times \h^{\wedge_0}$ (an isomorphism of formal Poisson schemes).
It follows that for a generic $\lambda\in \h$ there is a symplectic leaf $\mathcal{L}_1\subset X_\lambda$ of the same
dimension as $\mathcal{L}$ such that $\mathcal{L}\subset \overline{\C^\times \mathcal{L}_1}$
(the closure is taken in $X_{\h}$). The intersection of $\overline{\C^\times \mathcal{L}_1}$ (the closure is
taken in $Y_{\h}$) with $Y$ has therefore the same dimension as $\mathcal{L}$. It follows that for some
open leaf $\mathcal{L}'$ in the intersection of $\overline{\C^\times \mathcal{L}_1}$ with $X$,
we have $\dim \mathcal{L}'=\dim \rho(\mathcal{L}')$.


Since $\rho$ is a Poisson morphism, $\dim \mathcal{L}'=\dim \rho(\mathcal{L}')$   implies that there
is an open subset $\mathcal{L}^0\subset \mathcal{L}'$ that is an unramified cover of an open leaf $\underline{\mathcal{L}}$ in $\rho(\mathcal{L}')$. By the work of Namikawa, \cite{Namikawa_finite} (the case of open leaf) and
of Proudfoot and Schedler, the proof of \cite[Proposition 3.1]{PS} (the general case), the algebraic fundamental group of $\underline{\mathcal{L}}$ is finite. It follows that the algebraic fundamental group of $\mathcal{L}^0$ is finite. Since $\mathcal{L}'$ is smooth, we deduce that $\pi_1(\mathcal{L}^0)\twoheadrightarrow
\pi_1(\mathcal{L}')$ hence the algebraic fundamental group of $\mathcal{L}'$ is finite.

By Lemma  \ref{Lem:leaf_intersection}, the intersection of $\overline{\C^\times \mathcal{L}_1}$ with $X$ is irreducible,
so $\Leaf'=\Leaf$. This finishes the proof.
\end{proof}

\subsection{Proof of Theorem \ref{Thm:simplicity}}
We start with the following proposition.

\begin{Prop}\label{Prop:simplicity}
Suppose that the sheaf of algebras $\mathcal{D}_\lambda$ is simple for a Weil generic $\lambda\in \h$.
Then $\mathcal{D}_\lambda$ is simple for any $\lambda\in \h$.
\end{Prop}
\begin{proof}
Let $\mathcal{J}\subset \mathcal{D}_\lambda$ be a proper ideal. Let $\mathcal{L}\subset X$ be an open symplectic
leaf in the support of $\mathcal{D}_\lambda/\mathcal{J}_\lambda$. Pick a point $x\in \mathcal{L}$.
Let $\mathcal{J}_{\lambda,\hbar}$
be the two-sided ideal in $\mathcal{D}_{\lambda,\hbar}$ corresponding to $\mathcal{J}_\lambda$. Then
$\mathcal{D}_{\lambda,\hbar}^{\wedge_x}/ \mathcal{J}_{\lambda,\hbar}^{\wedge_x}$ is finitely generated over
$\Weyl_\hbar^{\wedge_0}(T_x\mathcal{L})$. On the other hand, by Lemma \ref{Rem:formal_quant}, $\mathcal{D}_{\h,\hbar}^{\wedge_x}=\C[[\h]]\widehat{\otimes} \mathcal{D}_{\lambda,\hbar}^{\wedge_0}$.
Set $\mathcal{I}_\hbar:=\C[[\h]]\widehat{\otimes} \mathcal{J}_{\lambda,\hbar}^{\wedge_x}$. By Proposition \ref{Prop:max_ideal}, we can find a $\C^\times$-stable ideal $\mathcal{J}_\hbar\subset \mathcal{D}_{\h,\hbar}$
such that $\mathcal{J}_\hbar^{\wedge_x}=\mathcal{I}_\hbar$.
Let $\mathcal{J}_{\hbar,fin}$ denote the $\C^\times$-finite part of $\mathcal{J}_{\hbar}$. Since
$\mathcal{D}_{\h,\hbar}^{\wedge_x}/\mathcal{I}_{\hbar}$ is flat over $\C[[\h]]$, we see that
$R_\hbar(\mathcal{D}_\h)/\mathcal{J}_{\hbar,fin}$ (that embeds into $\mathcal{D}_{\h,\hbar}^{\wedge_x}/\mathcal{I}_{\hbar}$) is torsion free over $\C[\h]$.  So for a Weil
generic $\lambda'\in \h$ the specialization of $\mathcal{D}_{\h,\hbar,fin}/\mathcal{J}_{\hbar,fin}$
to $\lambda'\hbar^d$ is nonzero. It follows that $\mathcal{D}_{\lambda'}$ is not simple, a contradiction
with Corollary \ref{Cor:Weil_gen_simpl}.
\end{proof}

\begin{proof}[Proof of Theorem \ref{Thm:simplicity}]
We know by \cite[Theorem 4.2.1]{sraco} that the algebra $H_c$ is simple for a Weil generic $c\in \param$.
It follows that $eH_c e$ is simple for such $c$. But $eH_ce=\A_\lambda$, where $\lambda$ is obtained from
$c$ by applying an affine isomorphism $\mathfrak{p}\xrightarrow{\sim}\mathfrak{h}$. So $\A_\lambda$
is simple.
By Corollary \ref{Cor:Weil_gen_simpl}, $\mathcal{D}_\lambda$
is simple for a Weil generic $\lambda$. Now we are done by Proposition \ref{Prop:simplicity}.
\end{proof}

Also let us record the following result.

\begin{Cor}\label{Cor:Conj_equiv_stat}
Suppose $\mathcal{D}_\lambda$ is simple. Then the support of every coherent $\mathcal{D}_\lambda$-module
intersects $X^{reg}$.
\end{Cor}
\begin{proof}
Let $M$ be a coherent $\mathcal{D}_\lambda$-module whose support does not intersect $X^{reg}$.
The sheaf of algebras $\mathcal{D}_\lambda$ is left Noetherian. It follows that $M$ has an irreducible quotient,
so we can assume $M$ itself is irreducible. As in the proof of  \cite[Theorem 1.1]{B_ineq},
Proposition \ref{Prop:max_ideal} implies that the support of $\mathcal{D}_\lambda/\operatorname{Ann}_{\mathcal{D}_\lambda}(M)$ is the closure of the single
leaf, $\mathcal{L}$, that is a maximal (with respect to inclusion) leaf such that
$\overline{\mathcal{L}}\cap \operatorname{Supp}(M)\neq \varnothing$. It follows that
$\operatorname{Ann}(M)$ is a proper 2-sided ideal. This finishes the proof.
\end{proof}

%

\section{Procesi sheaves}
Set $Y:=V/\Gamma$ and let $X$ be a $\Q$-factorial terminalization of $Y$.
In this section we axiomatically define and construct a Procesi sheaf on $X$.
In the case when $\Gamma$ is a so called {\it wreath-product}
group, and so $X$ is smooth, the construction was carried out by Bezrukavnikov and Kaledin, \cite{BK},
and our construction follows theirs.

\subsection{Definition of Procesi sheaf}
\begin{defi}\label{defi_Procesi}
A $\C^\times$-equivariant coherent sheaf $\Pro$ on $X$ together with an isomorphism
$\End(\Pro)\xrightarrow{\sim}\C[V]\#\Gamma$ is called a {\it Procesi sheaf}
if the following holds:
\begin{itemize}
\item[(i)] The isomorphism $\End(\Pro)\xrightarrow{\sim}\C[V]\#\Gamma$ is $\C^\times$-equivariant
and $\C[Y]$-linear.
\item[(ii)] We have $H^i(X, \mathcal{E}nd(\Pro))=0$  for $i>0$.
\item[(iii)] $\Pro^{\Gamma}\cong \mathcal{O}_{X}$, an isomorphism of $\C^\times$-equivariant
coherent sheaves.
\item[(iv)] $\mathcal{E}nd(\Pro)$ is a maximal Cohen-Macaulay $\mathcal{O}_{X}$-module.
\end{itemize}
\end{defi}

Note a few standard consequences of these conditions.
By (iii),  $\Pro=\mathcal{E}nd(\Pro)e$ and hence (iv) implies that $\Pro$ is a maximal
Cohen-Macaulay $\mathcal{O}_{X}$-module. In particular, $\Pro|_{X^{reg}}$ is a vector
bundle.   So when $X$ is smooth we recover an axiomatic description of a Procesi
bundle from \cite[Section 1.1]{Procesi}.

Let us also note that the definition makes sense for  fields different from $\C$ as well. For example,
as we discussed in Section \ref{SS_Q_fac_term},  we can reduce $X\rightarrow Y$ modulo $p$ for
$p\gg 0$ getting a $\Q$-factorial terminalization
$X_{\F}\rightarrow Y_{\F}$ for $\F:=\overline{\F}_p$. We can define a Procesi sheaf
$\Pro_{\F}$ on $X_{\F}$ similarly. In fact, we will need a Frobenius twisted version
$\Pro^{(1)}_{\F}$ on $X^{(1)}_{\F}$, which is again defined completely analogously.

\subsection{Frobenius constant quantization}
Our construction of $\Pro$ closely follows that of \cite{BK}, see Sections 5 and 6 there.
The first step is to produce a Frobenius-constant quantization of $X_{\F}$
(where $\F$ is as before) with a specified algebra of global sections. Let us start by
explaining what we mean by a Frobenius constant quantization in this context.

\begin{defi}\label{defi_Frob_const}
A Frobenius constant quantization $\A$ of $X_{\F}$ is a coherent sheaf of algebras on $X_{\F}^{(1)}$
whose restriction to the conical topology is equipped with a separated ascending filtration
such that $\gr\A\cong \Fr_*\mathcal{O}_{X_{\F}}$.
\end{defi}

Let us deduce several corollaries of this definition.

\begin{Lem}\label{Lem:Frob_const_properties}
Let $\A$ be a Frobenius constant quantization.
Then the  following is true:
\begin{enumerate}
\item $\A|_{X^{(1)reg}_{\F}}$ is an Azumaya algebra.
\item $\A$ is a maximal Cohen-Macaulay $\mathcal{O}_{X^{(1)}_{\F}}$-module.
\item $H^i(X_{\F}^{(1)},\A)=0$ for all $i>0$ and $\gr H^0(X^{(1)}_{\F},\A)=\F[V_{\F}]^\Gamma$.
\end{enumerate}
\end{Lem}
\begin{proof}
Let us prove (1). Let $\A_\hbar$ denote the $\hbar$-adic completion of the Rees sheaf of the filtered sheaf $\A$.
Then we have a central inclusion $\mathcal{O}_{X^{(1)}_\F}[[\hbar]]\hookrightarrow
\A_\hbar$. Now pick $x\in X^{(1),reg}_{\F}$ and consider the specialization $\A_{\hbar,x}$
to that point. It is a formal deformation of the  Frobenius neighborhood of $x$
in $X_{\F}$. The algebra of functions on the Frobenius neighborhood has no proper
Poisson ideals. Therefore the only proper two-sided ideals in $\A_{\hbar,x}$ are those
generated by $\hbar^k$ with $k>0$. It follows that the localization $\A_{\hbar,x}[\hbar^{-1}]$
is a simple algebra. Note that $\A_{\hbar,x}[\hbar^{-1}]\cong \A_x\otimes_{\F}\F((\hbar))$.
Since $\A_x\otimes_{\F}\F((\hbar))$ is simple, we see that the algebra $\A_x$ is simple as well.
So $\A_x$ is a matrix algebra of dimension $p^{\dim Y}$. This proves (1).

Let us prove (2). Since $p\gg 0$ and $X$ is Cohen-Macaulay, we see that
$X_{\F}$ is Cohen-Macaulay. It follows that $\operatorname{Fr}_*\mathcal{O}_{X_{\F}}$
is a maximal Cohen-Macaulay $\mathcal{O}_{X_{\F}^{(1)}}$-module. Since $\A$ is a filtered deformation
of a maximal Cohen-Macaulay module, it is maximal Cohen-Macaulay as well.

(3) follows from $H^i(X_{\F}, \mathcal{O}_{X_\F})=0$ for all $i>0$ and $\gr \A=\Fr_*\mathcal{O}_{X_{\F}}$.
\end{proof}

\begin{Prop}\label{Prop:quant}
There is a Frobenius constant quantization $\A$ of $X_{\F}$ such that
we have an isomorphism $H^0(X_{\F}^{(1)}, \A)=\Weyl(V_{\F})^\Gamma$
of filtered $\F[V^{(1)}_{\F}]^\Gamma$-algebras, where we write $\Weyl(V_{\F})$
for the Weyl algebra of $V_{\F}$.
\end{Prop}
\begin{proof}
Let $V^{(1),sr}_\F$ denote the open locus of $v\in V_\F^{(1)}$ such that
$\dim \left( V^{(1)}_\F\right)^{\Gamma_v}\geqslant \dim V-2$.
So $V^{(1),sr}_{\F}/\Gamma$ is the union of the open leaf and all codimension
$2$ symplectic leaves in $V^{(1)}_\F/\Gamma$. The preimage $X^{(1),sr}_{\F}$
of $V^{(1),sr}_{\F}/\Gamma$ in $X^{(1)}_{\F}$ consists of smooth points.
Moreover, the morphism $\rho_{\F}$ is semismall because, by Lemma \ref{Lem:ssmall}, $\rho$ is. Since
the complement to $V^{(1),sr}_{\F}/\Gamma$ has codimension at least $4$, and $\rho_\F$ is semismall,
it follows that
$$\operatorname{codim}_{X_{\F}^{(1)}}X_\F^{(1)}\setminus X_\F^{(1),sr}\geqslant 2.$$

Similarly to the proof of \cite[Proposition 5.10]{BK},  we get a Frobenius constant quantization
$\A^{sr}$ of $X^{sr}_{\F}$ (Frobenius constant quantization of $X^{sr}_{\F}$
are defined similarly to those  of $X_{\F}$) with $H^0(X^{(1)sr}_{\F},\A^{sr})=\Weyl(V_{\F})^{\Gamma}$
(an equality of filtered $\F[V_{\F}^{(1)}]^{\Gamma}$-algebras). Similarly to the proof
of \cite[Proposition 5.11]{BK}, we see that $\A^{sr}$ uniquely extends to  a Frobenius
constant quantization $\A^{reg}$ of $X_{\F}^{(1),reg}$. But $\gr \A^{reg}=
\operatorname{Fr}_* \mathcal{O}_{X^{reg}_{\F}}$ and the complement of
$X^{(1),reg}_{\F}$ in $X^{(1)}$ has codimension $4$.  Form here one deduces
that $\A:=i_*\A^{reg}$ is a  Frobenius constant quantization of
$X_{\F}$, compare to \cite[Proposition 3.4]{BPW}. By the construction, $\A$ has required properties.
\end{proof}

\subsection{Construction of Procesi sheaf in characteristic $p$}
Consider the schemes $\hat{V}_{\F}^{(1)}$ and  $\hat{Y}^{(1)}_{\F}$,
the spectra of the completions of    $\F[V^{(1)}]$ and $\F[Y^{(1)}]$ at $0$. Set
$$\hat{X}^{(1)}_{\F}:=\hat{Y}^{(1)}_{\F}\times_{Y^{(1)}_{\F}}X^{(1)}_{\F},\hat{\A}:=\A|_{\hat{X}^{(1)}_{\F}},
\hat{\Weyl}:=\F[\hat{V}_{\F}^{(1)}]\otimes_{\F[V_{\F}^{(1)}]}\Weyl(V_{\F}).$$
So we have $R\Gamma(\hat{\A})=\hat{\Weyl}^\Gamma$.

Set $$\hat{\A}^{reg}:=\hat{\A}|_{X^{(1)reg}_{\F}\cap \hat{X}^{(1)}_{\F}}.$$
By (1) of Lemma \ref{Lem:Frob_const_properties}, $\hat{\A}^{reg}$ is an Azumaya algebra.
Similarly to \cite[Section 6.3]{BK}, we see that $\hat{\A}^{reg}$
splits, while $\hat{\Weyl}|_{\hat{V}^{(1)}_{\F}}$
$\Gamma$-equivariantly splits (here we write $V^{(1)r}_{\F}$ for the locus in
$V^{(1)}_{\F}$ consisting of all points with trivial stabilizer in $\Gamma$).

Now let us define a sheaf of algebras $\underline{\hat{\A}}$ on $\hat{X}^{(1)}_{\F}$
that is Morita equivalent to $\hat{\A}$ and has global sections $\F[\hat{V}^{(1)}_{\F}]\#\Gamma$.
Similarly to \cite[Sections 6.1, 6.3]{BK}, the algebras $\hat{\Weyl}^{\Gamma}$ and $\F[\hat{V}^{(1)}_{\F}]\#\Gamma$
are Morita equivalent. It follows that we can find central idempotents $e_i\in \hat{\Weyl}^\Gamma$,
one per an irreducible representation of $\Gamma$ and nonzero multiplicities $n_i$
such that $\F[\hat{V}^{(1)}_{\F}]\#\Gamma=\bigoplus_{i,j} (e_i \hat{\Weyl}^\Gamma e_j)^{n_in_j}$.
Set $\underline{\hat{\A}}:=\bigoplus_{i,j} (e_i \hat{\A} e_j)^{n_in_j}$.

\begin{Lem}\label{Lem:hat_underline_prop}
The sheaf $\underline{\hat{\A}}$ on $\hat{X}^{(1)}_{\F}$ has the following properties.
\begin{enumerate}
\item $H^0(\hat{X}^{(1)}_{\F},\underline{\hat{\A}})=\F[\hat{V}^{(1)}_{\F}]\#\Gamma$.
\item $H^i(\hat{X}^{(1)}_{\F},\underline{\hat{\A}})=0$ for $i>0$.
\item $\underline{\hat{\A}}$ is a maximal Cohen-Macaulay $\mathcal{O}_{\hat{X}^{(1)}_{\F}}$-module.
\item $\underline{\hat{\A}}^{reg}:=\underline{\hat{\A}}|_{\hat{X}_\F^{(1)}\cap X^{(1)reg}_{\F}}$ is a split Azumaya algebra.
\item Let $e$ denote the averaging idempotent in $\F\Gamma$. Then $e\underline{\hat{\A}}e=
\mathcal{O}_{\hat{X}}$.
\item The sheaf of algebras $\underline{\hat{\A}}$ coincides with the endomorphism sheaf of
$\hat{\Pro}:=\underline{\hat{\A}}e$.
\end{enumerate}
\end{Lem}
\begin{proof}
(1)-(3) follow from the construction of $\underline{\hat{\A}}$ and the analogous
properties of $\A$, see Lemma \ref{Lem:Frob_const_properties}. (4) follows
from the discussion above in this section.

Let us prove (5). We first claim that the generic rank of $e\underline{\hat{\A}}e$ is $1$.
Indeed, the restriction of $e\underline{\hat{\A}}e$ to $\hat{X}^{(1)}_{\F}\cap \rho^{-1}_\F(Y_\F^{(1)reg})$
coincides with the restriction of $e' \Weyl(V_\F)^\Gamma e'$, where $e'$ is a primitive
idempotent corresponding to the trivial representation of $\Gamma$. The latter restriction
is easily seen to be the structure sheaf. Since $e\underline{\hat{\A}}^{reg}e$ is a split
Azumaya algebra, it is the structure sheaf.
Since
both $\mathcal{O}_{\hat{X}^{(1)}_\F}$ and $e\underline{\hat{\A}}e$ are maximal Cohen-Macaulay
and codimension of the singular locus is at least $4$,
we see that $\mathcal{O}_{\hat{X}^{(1)}_\F} \xrightarrow{\sim} e\underline{\hat{\A}}e$.

Let us prove (6). By the construction, $\hat{\Pro}$ is a maximal Cohen-Macaulay
$\mathcal{O}_{\hat{X}^{(1)}_\F}$-module. So $\hat{\Pro}=i_*i^* \hat{\Pro}$, where we write
$i$ for the inclusion $\hat{X}^{(1)reg}_\F\hookrightarrow \hat{X}^{(1)}_\F$.
It follows that $\mathcal{E}nd(\hat{\Pro})=i_*\mathcal{E}nd(i^*\hat{\Pro})$.
Similarly, $\underline{\hat{\A}}=i_*i^*\underline{\hat{\A}}$. And since
$i^*\underline{\hat{\A}}$ is a split Azumaya algebra with $e(i^*\underline{\hat{\A}})e\cong
\mathcal{O}_{\hat{X}^{(1)reg}}$, we have $i^*\underline{\hat{\A}}=
\mathcal{E}nd(i^*\hat{\Pro})$. This implies (6).
\end{proof}

So the sheaf $\hat{\Pro}$ behaves almost like a Procesi sheaf with two differences:
it does not carry an action of $\F^\times$ yet and it is defined on $\hat{X}^{(1)}_{\F}$,
while we originally wanted a sheaf on $X$. For this we first equip it with an
$\F^\times$-equivariant structure and extend to a sheaf $\Pro_{\F}^{(1)}$ $X_{\F}^{(1)}$.

\begin{Lem}\label{Lem:ext1}
There is a unique Procesi sheaf $\Pro_{\F}^{(1)}$ on $X_{\F}^{(1)}$ whose restriction to
$\hat{X}^{(1)}_{\F}$ coincides with $\hat{\Pro}$.
\end{Lem}
\begin{proof}
The proof is in several steps.

{\it Step 1}. Let us equip $\hat{\Pro}$ with an $\F^\times$-equivariant structure.  Our proof follows \cite{Vologodsky} with some modifications.

Let us introduce some notation first. We write $Y$ for
$\hat{Y}_\F^{(1)}$ and $X$ for $\hat{X}_\F^{(1)}$. Let $\mathfrak{m}$ denote the maximal
ideal in $\F[Y]$. For $\ell>0$, we write $Y_\ell$ for $\operatorname{Spec}(\F[Y]/\mathfrak{m}^\ell)$
and $X_\ell$ for $Y_\ell\times_Y X$. We write $m,p_2$ for the multiplication and second projection
morphisms $\F^\times\times X\rightarrow X, \F^\times\times X_\ell\rightarrow X_\ell$. Finally, for a coherent
sheaf $\mathcal{F}$ on $\F^\times\times X$, we write
$\mathcal{F}^\wedge$ for the restriction of $\mathcal{F}$ to the formal neighborhood of
$\{1\}\times X$ in $\F^\times \times X$. This is an exact functor. We use the same notation
for coherent sheaves on $\F^\times\times X_\ell$.

Consider the sheaves $m^* \hat{\Pro}, p_2^* \hat{\Pro}$ on $\F^\times\times X$.  Note that
\begin{equation}\label{eq:restric_iso}(m^* \hat{\Pro})^\wedge\cong (p_2^* \hat{\Pro})^\wedge.\end{equation}
Indeed, $\operatorname{Ext}^1(i^* \hat{\Pro},i^*\hat{\Pro})=0$
so the restrictions of $m^* i^* \hat{\Pro}, p_2^* i^* \hat{\Pro}$ to the formal neighborhood of $\{1\}\times X^{reg}$
are isomorphic. (\ref{eq:restric_iso}) follows from this isomorphism thanks to $\hat{\Pro}=i_*i^* \hat{\Pro}$.

{\it Step 2}.
Let us write $\hat{\Pro}_\ell$ for the restriction of $\hat{\Pro}$ to $X_\ell$. It follows from
(\ref{eq:restric_iso}) that
\begin{equation}\label{eq:restric_iso1}(m^* \hat{\Pro}_\ell)^\wedge\cong (p_2^* \hat{\Pro}_\ell)^\wedge.\end{equation}
Consider the coherent sheaf $\mathcal{M}_\ell:=\mathcal{H}om(m^* \hat{\Pro}_\ell, p_2^* \hat{\Pro}_\ell)$
on $\F^\times \times X_\ell$. We have
\begin{equation}\label{eq:compl_iso_equiv}\mathcal{M}_\ell^\wedge=\mathcal{H}om((m^* \hat{\Pro}_\ell)^{\wedge}, (p_2^* \hat{\Pro}_\ell)^\wedge).\end{equation}
Thanks to (\ref{eq:restric_iso1}), the push-forward of the right hand side of (\ref{eq:compl_iso_equiv})
to the formal neighborhood of $1$ in $\F^\times$ is a free module. Now we use the formal function theorem
for the projection $p_1:\F^\times \times X_\ell\rightarrow X_\ell$ together with   (\ref{eq:compl_iso_equiv})
to conclude that $p_{1*}\mathcal{M}$ is a vector bundle on some neighborhood of $1$ in $\F^\times$.
As in the proof of \cite[Lemma 6.1]{Vologodsky}, this implies $m^* \hat{\Pro}_\ell\cong p_2^* \hat{\Pro}_\ell$
for all $\ell$.
Now we can argue as in the last paragraph of the proof of \cite[Proposition 6.3]{Vologodsky},
to conclude that $\hat{\Pro}$ has an $\F^\times$-equivariant structure.

{\it Step 3}. Let us show that we can choose an $\F^\times$-equivariant structure on
$i^*\hat{\Pro}$ in such a way that the isomorphism $\End(i^*\hat{\Pro})\xrightarrow{\sim}
\F[[V^{(1)}_\F]]\#\Gamma$ is $\F^\times$-equivariant.  Pick a basis in each irreducible representation of $\Gamma$
and let $\mathcal{B}$ denote the union of these bases. The algebra
$\End(i^*\hat{\Pro})/(\mathfrak{m})$ is finite dimensional. So, for each $v\in \mathcal{B}$,
we can choose an $\F^\times$-invariant primitive idempotent in this algebra corresponding
to $v$. Then we can lift that idempotent to an $\F^\times$-invariant idempotent
in $\F[[V^{(1)}_\F]]\#\Gamma$, denote the resulting idempotent by $e_v$.

 We  decompose $i^*\hat{\Pro}$ as $\bigoplus_{v\in \mathcal{B}} \hat{\Pro}^{reg}_v$,
where  $\hat{\Pro}^{reg}_v:=e_v i^*\hat{\Pro}$
is an indecomposable vector bundle on $\hat{X}^{(1)reg}_{\F}$.  Each
$\hat{\Pro}^{reg}_v$ is $\F^\times$-stable. We note that $H^0(\hat{\Pro}^{reg}_v)=
e_v \F[[V^{(1)}]]$. The $\F[[V^{(1)}]]^\Gamma$-module $e_v \F[[V^{(1)}]]$ is indecomposable,
so an $\F^\times$-equivariant structure on $e_v \F[[V^{(1)}]]$
is unique up to a twist with a character. So twisting the $\F^\times$-equivariant structure on each
$\Pro^{reg}_v$ we achieve that the isomorphism  $\End(i^*\hat{\Pro})\xrightarrow{\sim}
\F[[V^{(1)}_\F]]\#\Gamma$ is $\F^\times$-equivariant.

{\it Step 4}. Since the action of $\F^\times$ on $X^{(1)}_{\F}$ is contracting and $X^{(1)}_{\F}$
is projective over an affine scheme, we see that $\hat{\Pro}_{\F}$ extends to a unique
$\F^\times$-equivariant sheaf  $\Pro_{\F}^{(1)}$ on $X_{\F}^{(1)}$, see \cite[Section 2.3]{BK}.
Note that $\mathcal{E}nd(\Pro_{\F}^{(1)})$
is a unique $\F^\times$-equivariant extension of $\mathcal{E}nd(\hat{\Pro}_{\F})$. Now to check
that $\Pro^{(1)}_{\F}$ is a Procesi sheaf is straightforward.
\end{proof}

So we have got a Procesi bundle $\Pro_{\F}^{(1)}$ on $X_\F^{(1)}$. We carry $\Pro_\F^{(1)}$ to $X_\F$
using the natural $\operatorname{Spec}\mathbb{Z}$-scheme isomorphism $X_\F\cong X_\F^{(1)}$.
The resulting vector bundle, to be denoted by $\Pro_\F$, is a Procesi bundle.

\subsection{Construction of Procesi sheaf in characteristic $0$}
We can lift the Procesi sheaf to characteristic $0$. We follow an argument in
\cite[Section 2.3]{BK} with some modifications.
Namely, $X_\F, \Pro_\F$ are defined over some finite field $\F_q$, let $X_{\F_q}, \Pro_{\F_q}$ be the corresponding
reductions. Let $R$ denote an integral extension of $\Z_p$ with residue field $\F_q$. We can assume that
$X$ is defined over $R$, let $X_R$ denote the corresponding form so that $X_{\F_q}=\operatorname{Spec}(\F_q)\times_{\operatorname{Spec}(R)}X_R$.

\begin{Lem}\label{Lem:Procesi_R_extension}
$\Pro_{\F_q}$ uniquely extends to a Procesi bundle $\Pro_R$ on $X_R$.
\end{Lem}
\begin{proof}
Consider the open subscheme $X_{\F_q}^{reg}\subset X_{\F_q}$, the complement still has codimension
at least $4$. Let $X_R^{\wedge}$ denote the formal neighborhood of $X_{\F_q}$ in $X_R$ and let
$X_R^\circ$ be the formal neighborhood of $X_{\F_q}^{reg}$ in $X_R$ so that we have an inclusion
of formal schemes $\hat{\iota}:X_R^\circ\hookrightarrow X_R$. So $X_R^\circ$ is a formal deformation of
$X_{\F_q}^{reg}$.

Note that $\Ext^i(i^*\Pro_{\F_q},i^*\Pro_{\F_q})=0$ for $i=1,2$ because $\mathcal{E}nd(\Pro_{\F_q})$
is maximal Cohen-Macaulay and $\operatorname{codim}_{X_{\F_q}}X_{\F_q}^{sing}\geqslant 4$.
So $i^*\Pro_{\F_q}$ admits a unique $\mathbb{G}_m$-equivariant deformation to
$X_R^\circ$, let us denote it by $\Pro_R^\circ$.
Also the  condition $\operatorname{codim}_{X_{\F_q}}X_{\F_q}^{sing}\geqslant 4$ and the
fact that $\Pro_{\F_q}$ is maximal Cohen-Macaulay imply $(R^1i_*)i^*\Pro_{\F_q}=0$.
So $\Pro_R^{\wedge}:=\hat{\iota}_*\Pro_R^\circ$ is an
$R$-flat deformation of $i_*i^*\Pro_{\F_q}=\Pro_{\F_q}$. For a similar reason,
$\mathcal{E}nd(\Pro_R^\wedge)$ is an $R$-flat deformation of $\mathcal{E}nd(\Pro_{\F_q})$.
Because of the $\mathbb{G}_m$-equivariance we can extend $\Pro_R^{\wedge}$ to a unique $\mathbb{G}_m$-equivariant
coherent  sheaf $\Pro_R$ on $X_R$. That $\End(\Pro_R)$ is isomorphic to $R[V_R]\#\Gamma$  is proved similarly
to \cite[Section 6.4]{BK}, as the assumptions of \cite[Proposition 4.3]{BK} hold in our case for the same
reason they hold in \cite{BK}. The other axioms of the Procesi sheaves for $\Pro_R$
are straightforward.
\end{proof}

Then we can  change the base from $R$ to $\C$ and get a required Procesi sheaf on $X$.

\begin{Rem}\label{Rem:Procesi_classif}
One can also classify all Procesi bundles on $X$ similarly to \cite{Procesi}: they are
in bijection with the Namikawa-Weyl group of $Y$. The proof basically repeats
\cite[Sections 2,3]{Procesi}.
\end{Rem}

\section{Derived equivalences from Procesi sheaves}
In this section we are going to use the Procesi sheaf $\Pro$ on $X$ to prove Theorems \ref{Thm:der_equiv1}
and \ref{Thm:der_equiv2}. We set $\mathcal{H}^0:=\mathcal{E}nd(\Pro)$, this is
a maximal Cohen-Macaulay sheaf on $X$.

\subsection{Quantizations of the Procesi sheaf}
Let $\mathcal{D}_c$ be the quantization of $X$ corresponding to $c\in \param$. Let us write
$\mathcal{D}_c^{reg}$ for the restriction of $\mathcal{D}_c$ to $X^{reg}$ and $\Pro^{reg}$
for the restriction of $\Pro$ to $X^{reg}$. As in the proof of Lemma \ref{Lem:Procesi_R_extension},
we see that $\Ext^i(\Pro^{reg},\Pro^{reg})=0$
for $i=1,2$. It follows that $\Pro^{reg}$ admits a unique deformation to  a locally free
right $\mathcal{D}_c^{reg}$-module to be denoted by $\Pro^{reg}_c$. As usual, we set
$\Pro_c:=i_*\Pro^{reg}_c$, where $i$ stands for the inclusion $X^{reg}\hookrightarrow X$.
We can also define the universal version $\Pro_{\param}^{reg}$
on $X^{reg}$ and its push-forward $\Pro_{\param}:=i_* \Pro^{reg}_{\param}$. Let us write
$\mathcal{H}_c$ for $\mathcal{E}nd_{\mathcal{D}_c^{opp}}(\Pro_c)$.

\begin{Prop}\label{Prop:quant_Procesi_endom}
The following claims are true:
\begin{enumerate}
\item $\gr\mathcal{H}_c=\mathcal{H}^0$ for all $c$.
\item $\Gamma(\mathcal{H}_c)=H_{w(c)}$, where $w$ is an element of the Namikawa-Weyl group
depending only on $\Pro$.
\end{enumerate}
\end{Prop}
\begin{proof}
Let us prove (1). Consider the $\hbar$-adically completed Rees sheaf $\Pro_{c,\hbar}$ of $\Pro_c$ so that we have an
exact sequence $0\rightarrow \Pro_{c,\hbar}\xrightarrow{\hbar\cdot}\Pro_{c,\hbar}\rightarrow \Pro\rightarrow 0$.
We need to prove that $\mathcal{E}nd(\Pro_{c,\hbar})/(\hbar)=\mathcal{H}^0$. We know from the
construction that $i^*\mathcal{E}nd(\Pro_{c,\hbar})/(\hbar)= i^*\mathcal{H}^0$. Also we have $i_*i^*\mathcal{H}^0=
\mathcal{H}^0$. From the inclusion $\mathcal{E}nd(\Pro_{c,\hbar})/(\hbar)\hookrightarrow
\mathcal{H}^0$  we conclude that the natural homomorphism $\mathcal{E}nd(\Pro_{c,\hbar})/(\hbar)\rightarrow
i_*i^* \left(\mathcal{E}nd(\Pro_{c,\hbar})/(\hbar)\right)$ is an isomorphism.
It follows that $\mathcal{E}nd(\Pro_{c,\hbar})\rightarrow
i_*i^* \mathcal{E}nd(\Pro_{c,\hbar}))$ is an isomorphism.

What remains  to prove to establish (1) is that
$R^1i_* i^*\mathcal{E}nd(\Pro_{c,\hbar})=0$.  We have $R^1i_* i^*\mathcal{H}^0=0$
because $\mathcal{H}^0$ is maximal Cohen-Macaulay and the codimension of the complement
of $X^{reg}$ is, at least, $4$. Similarly to the proof of
\cite[Lemma 5.6.3]{GL}, it follows that $R^1i_* i^*\mathcal{E}nd(\Pro_{c,\hbar})=0$.

Let us prove (2). Since $\Ext^1(\Pro,\Pro)=0$ (compare to the proof of Lemma \ref{Lem:Procesi_R_extension}),
we see that $\gr \End(\Pro_\param)=\gr H_{\param}$.
The universality property of $H_{\param}$, see \cite[Section 6.1]{quant_iso}, implies that there is a filtered algebra isomorphism
$\End(\Pro_{\param})\xrightarrow{\sim} H_{\param}$ that induces an affine map on $\param$
and gives the identity map on $\C[V]\#\Gamma$.  Passing to spherical subalgebras we get a filtered algebra
automorphism $e H_{\param} e\xrightarrow{\sim} e H_{\param} e$ that induces the identity automorphism
of $\C[V]^\Gamma$. This isomorphism has to be given by an element of the Namikawa-Weyl group because
$(e H_{\param} e)^W$ is a universal filtered quantization of $\C[V]^\Gamma$ as discussed in
Section \ref{SS_quant}.
\end{proof}

\subsection{McKay equivalence}
In this section we will prove the following result. Consider the category
$\Coh(\mathcal{H}^0)$ of coherent $\mathcal{H}^0$-modules
on $X$.

\begin{Prop}\label{Prop:McKay}
Let $\Pro$ be a Procesi bundle on $X$. Then the derived global section functor
$R\Gamma: D^b(\Coh(\mathcal{H}^0))\xrightarrow{\sim} D^b(\C[V]\#\Gamma\operatorname{-mod})$
is a category equivalence.
\end{Prop}

Consider the full triangulated subcategory $\mathcal{C}$ of $D^b(\Coh(\mathcal{H}^0))$
generated by objects of the form $\mathcal{H}^0\otimes \mathcal{L}$, where $\mathcal{L}$
is a line bundle on $X$.
Similarly to \cite[Section 2.2]{BK}, a crucial step in the proof of  Proposition
\ref{Prop:McKay} is to prove the following lemma.

\begin{Lem}\label{Lem:indec_Serre}
The following claims are true:
\begin{enumerate}
\item The category $\mathcal{C}$ is indecomposable (as a triangulated category).
\item The functor $\bullet[\dim X]$ is a Serre functor for $\mathcal{C}$
meaning that $$R\Hom_{\C[Y]}\left(R\Hom_{\mathcal{H}^0}(\mathcal{F},\mathcal{G}),\C[Y][\dim Y]\right)\cong
R\Hom_{\mathcal{H}^0}(\mathcal{G}, \mathcal{F}[\dim X]).$$
\end{enumerate}
\end{Lem}
\begin{proof}
Let us prove (1). Let $\epsilon$ be a primitive idempotent in $\C\Gamma$. Then, for $\mathcal{F}\in \operatorname{Coh}(X)$,
$\mathcal{H}om(\mathcal{H}^0\epsilon\otimes \mathcal{L},\mathcal{F})=\mathcal{L}^*\otimes \epsilon\mathcal{F}$, an equality
of coherent sheaves on $X$. So
$\Hom(\mathcal{H}^0\epsilon\otimes \mathcal{L},\mathcal{F})=\Gamma(\mathcal{L}^*\otimes \epsilon\mathcal{F})$.
This space is nonzero as long as $e\mathcal{F}$ is nonzero and $\mathcal{L}$ is sufficiently anti-ample.

Note that the object $\mathcal{H}^0\epsilon\otimes \mathcal{L}$ is indecomposable. Indeed, it is enough to assume that
$\mathcal{L}$ is trivial. The global sections of $\mathcal{H}^0\epsilon$ is $(\C[V]\#\Gamma)\epsilon$.
This is an indecomposable
$\C[V]\#\Gamma$-module because $\epsilon$ is primitive. So if $\mathcal{H}^0\epsilon$ is
decomposable, then one of its summands
has zero global sections, equivalently push-forward to $Y$.
Since the morphism $X\rightarrow Y$ is an isomorphism generically, that summand
 must have proper support in $X$. The latter is impossible because
$\mathcal{H}^0$ is maximal Cohen-Macaulay. So we see that $\mathcal{H}^0\epsilon\otimes \mathcal{L}$ is indecomposable. If we can decompose $\mathcal{C}$ into the direct sum of two triangulated categories,
$\mathcal{C}_1\oplus \mathcal{C}_2$, then for each $e,\mathcal{L}$ the object $\mathcal{H}^0\epsilon\otimes \mathcal{L}$
lies in one summand.

Now note that $\Hom(\mathcal{H}^0\epsilon, \mathcal{H}^0\epsilon')=\epsilon(\C[V]\#\Gamma)\epsilon'\neq 0$. So
all $\mathcal{H}^0\epsilon$ lie in the same summand, say $\mathcal{C}_1$. It follows that for any
sufficient anti-ample $\mathcal{L}$, we have $\mathcal{H}^0\epsilon\otimes \mathcal{L}\in \mathcal{C}_1$.
So $\mathcal{C}_2=0$ and the proof of (1) is finished.

Let us prove (2). Recall, (1) of Lemma \ref{Lem:X_properties}, that $K_X$ is trivial. As in  \cite[Section 2.1]{BK},
since the Grothendieck-Serre duality commutes with proper direct images,     what we need to check is
\begin{equation}\label{eq:sheaf_dual} R\mathcal{H}om_{\mathcal{O}_X}\left(R\mathcal{H}om_{\mathcal{H}^0}(\mathcal{F},\mathcal{G}),
\mathcal{O}_X[\dim X]\right)\cong R\mathcal{H}om_{\mathcal{H}^0}(\mathcal{G}, \mathcal{F}[\dim X]),
\end{equation}
for $\mathcal{F},\mathcal{G}$ in $\mathcal{C}$.
Since the objects $\mathcal{H}^0\otimes \mathcal{L}$ generate $\mathcal{C}$,
it is enough to establish a natural isomorphism
of the left hand side and the right hand side of (\ref{eq:sheaf_dual})
when $\mathcal{F}=\mathcal{H}^0\otimes \mathcal{L}, \mathcal{G}=\mathcal{H}^0\otimes \mathcal{L}'$.
In particular, the isomorphism has to be bilinear with respect to $\End_{\mathcal{H}^0}(\mathcal{H}^0)=\End(\Pro)$.

The left hand side becomes
\begin{equation}\label{eq:sheaf1}
\begin{split}
&R\mathcal{H}om_{\mathcal{O}_X}(\mathcal{H}^0\otimes (\mathcal{L}\otimes \mathcal{L}'^*),\mathcal{O}_X[\dim X])=\\
&R\mathcal{H}om_{\mathcal{O}_X}(\mathcal{L}\otimes \mathcal{L}'^*,R\mathcal{H}om_{\mathcal{O}_X}(\mathcal{H}^0,\mathcal{O}_X)[\dim X]),
\end{split}
\end{equation}
while the right hand side is
\begin{equation}\label{eq:sheaf2}
\begin{split}
&R\mathcal{H}om_{\mathcal{H}^0}(\mathcal{H}^0\otimes (\mathcal{L}\otimes\mathcal{L}'^*),\mathcal{H}^0[\dim X])=\\&R\mathcal{H}om_{\mathcal{O}_X}(\mathcal{L}\otimes\mathcal{L}'^*,\mathcal{H}^0[\dim X]),
\end{split}
\end{equation}
So we need to establish an $\mathcal{H}^0$-bilinear isomorphism
\begin{equation}\label{eq:bilin_iso}
R\mathcal{H}om_{\mathcal{O}_X}(\mathcal{H}^0,\mathcal{O}_X)\cong \mathcal{H}^0.\end{equation}
Recall that $\mathcal{H}^0$ is a maximal Cohen-Macaulay module and so the left hand side of
(\ref{eq:bilin_iso}) is concentrated in homological degree $0$ and is a maximal Cohen-Macaulay module. So we just need
to establish a bilinear isomorphism $\mathcal{H}^0\cong \mathcal{H}^{0*}$. Since both side
are maximal Cohen-Macaulay modules, it is enough to establish an isomorphism on $X^{reg}$. But $\Pro|_{X^{reg}}$
is a vector bundle and $\mathcal{H}^0|_{X^{reg}}\cong \Pro|_{X^{reg}}\otimes
\left(\Pro|_{X^{reg}}\right)^*$. This implies $\mathcal{H}^0\cong \mathcal{H}^{0*}$
and finishes the proof.
\end{proof}

\begin{proof}[Proof of Proposition \ref{Prop:McKay}]
The proof closely follows that of \cite[Proposition 2.2]{BK}. Namely, we have the left adjoint functor
$\mathcal{H}^0\otimes^L_{\C[V]\#\Gamma}\bullet$ of $R\Gamma: D^b(\Coh(\mathcal{H}^0))
\rightarrow D^b(\C[V]\#\Gamma\operatorname{-mod})$. The left adjoint functor
is right inverse to $R\Gamma$ because of
$R\Gamma(\mathcal{H}^0)=\C[V]\#\Gamma$. We need to prove that $\mathcal{H}^0\otimes^L_{\C[V]\#\Gamma}\bullet$
is essentially surjective. Since $\bullet[\dim X]$ is a Serre functor for $\mathcal{C}$ and the category $\mathcal{C}$ is indecomposable, we can apply \cite[Lemma 2.7]{BK} and  conclude that the essential image of
$\mathcal{H}^0\otimes^L_{\C[V]\#\Gamma}\bullet$ is $\mathcal{C}$. But the right orthogonal of
$\mathcal{C}$ in $D^b(\operatorname{Coh}(\mathcal{H}^0))$ is zero. This is because
$$\operatorname{Hom}_{\mathcal{H}^0}(\mathcal{H}^0\otimes \mathcal{L}[-i],\mathcal{F})=H^i(\mathcal{L}^*\otimes \mathcal{F}).$$
If the right hand side vanishes for all $i$ and all sufficiently anti-ample $\mathcal{L}$, then $\mathcal{F}=0$.
Since the right orthogonal of $\mathcal{C}$ is zero, the proof of \cite[Lemma 2.7]{BK} shows that
$\mathcal{C}=D^b(\operatorname{Coh}(\mathcal{H}^0))$.
\end{proof}

\subsection{Proofs of Theorems \ref{Thm:der_equiv1},\ref{Thm:der_equiv2}}
We  prove Theorem \ref{Thm:der_equiv2} and then deduce Theorem \ref{Thm:der_equiv1} from here.
Our first step is the following corollary that follows from Proposition
\ref{Prop:McKay} similarly to what was done in \cite[Section 5.5]{GL}.

\begin{Cor}\label{Cor:quant_McKay}
The derived global section functor
$R\Gamma: D^b(\Coh(\mathcal{H}_c))\xrightarrow{\sim} D^b(H_{w(c)}\operatorname{-mod})$
is a category equivalence for all $c$.
\end{Cor}

Theorem \ref{Thm:der_equiv2} is now a consequence of the following proposition.

\begin{Prop}\label{Prop:spher_equiv_local}
The functor $M\mapsto eM: \operatorname{Coh}(\mathcal{H}_c)\rightarrow \operatorname{Coh}(\mathcal{D}_c)$
is a category equivalence.
\end{Prop}
\begin{proof}
Suppose that
\begin{enumerate}
\item
abelian localization holds for $(X,c)$
\item and  the parameter $w(c)$ (where $w\in W_Y$
is the element defined by $\Pro$, see (2) of Proposition \ref{Prop:quant_Procesi_endom}) is spherical.
\end{enumerate}
Note that the global section functors $\Gamma^{\mathcal{H}}: \operatorname{Coh}(\mathcal{H}_c)\rightarrow H_{w(c)}\operatorname{-mod}$ and $\Gamma^{\mathcal{D}}: \operatorname{Coh}(\mathcal{D}_c)\rightarrow eH_{w(c)}e\operatorname{-mod}$ satisfy $e\Gamma^{\mathcal{H}}(\bullet)\cong
\Gamma^{\mathcal{D}}(e\bullet)$. The functors $\Gamma^{\mathcal{D}}$ and $N\mapsto eN:
H_{w(c)}\operatorname{-mod}\rightarrow eH_{w(c)}e\operatorname{-mod}$ are category equivalences.
We conclude that the functor $\Gamma^{\mathcal{H}}$ is exact. By Corollary \ref{Cor:quant_McKay},
the functor $R\Gamma^{\mathcal{H}}$ is an equivalence, hence $\Gamma^{\mathcal{H}}$ is an equivalence.
It follows that the functor $M\mapsto eM: \operatorname{Coh}(\mathcal{H}_c)\rightarrow \operatorname{Coh}(\mathcal{D}_c)$
is an equivalence under assumptions (1),(2).

Now let us prove the claim of the  proposition without restrictions on $c$.
Recall that, for $\chi\in \operatorname{Pic}(X)$,  we have the $\mathcal{D}_{c+\chi}$-$\mathcal{D}_c$-bimodule $\mathcal{D}_{c,\chi}$ quantizing the line bundle $\mathcal{O}(\chi)$ and also the
$\mathcal{D}_{c}$-$\mathcal{D}_{c+\chi}$-bimodule $\mathcal{D}_{c+\chi,-\chi}$.
Tensor products with these bimodules define mutually quasi-inverse equivalences
between $\Coh(\mathcal{D}_c)$ and $\Coh(\mathcal{D}_{c+\chi})$.
In particular, we have natural isomorphisms
$$\mathcal{H}_c\cong
\mathcal{E}nd_{\mathcal{D}_{c+\chi}^{opp}}(\Pro_c\otimes_{\mathcal{D}_c}\mathcal{D}_{c+\chi,-\chi}),
\mathcal{H}_{c+\chi}\cong
\mathcal{E}nd_{\mathcal{D}_{c}^{opp}}(\Pro_{c+\chi}\otimes_{\mathcal{D}_{c+\chi}}\mathcal{D}_{c,\chi}).$$
So we can consider $\mathcal{H}_{c+\chi}$-$\mathcal{H}_c$-bimodule $\mathcal{E}_{c,\chi}:=\mathcal{H}om_{\mathcal{D}_{c+\chi}^{opp}}(\Pro_{c}\otimes_{\mathcal{D}_c}\mathcal{D}_{c+\chi,-\chi},
\Pro_{c+\chi})$ and the $\mathcal{E}nd(\Pro_{c})$-$\mathcal{E}nd(\Pro_{c+\chi})$-bimodule
$\mathcal{E}_{c+\chi,-\chi}:=
\mathcal{H}om_{\mathcal{D}_c^{opp}}(\Pro_{c+\chi}\otimes_{\mathcal{D}_{c+\chi}}\mathcal{D}_{c,\chi},\Pro_{c})$.

We claim that the bimodules $\mathcal{E}_{c,\chi}$ and $\mathcal{E}_{c+\chi,-\chi}$ are mutually inverse
Morita equivalences. Indeed, these bimodules come with natural filtrations and
$\gr\mathcal{E}_{c,\chi}\cong \mathcal{H}^0\otimes \mathcal{O}(\chi)$ and
$\gr \mathcal{E}_{c+\chi,-\chi}\cong \mathcal{H}^0\otimes
\mathcal{O}(-\chi)$. The latter $\mathcal{H}^0$-bimodules are locally trivial.
Tensoring over $\mathcal{H}^0$ with these associated graded bimodules
give autoequivalences of the category $\Coh(\mathcal{H}^0)$. Since the associated graded
bimodules $\gr\mathcal{E}_{c,\chi}, \gr \mathcal{E}_{c+\chi,-\chi}$
are mutually inverse Morita equivalence bimodules, so are the bimodules $\mathcal{E}_{c,\chi}, \mathcal{E}_{c+\chi,-\chi}$.

For any $c$, we can choose $\chi$ so that abelian localization holds for $(X,c+\chi)$. Note that
$\Pro_{c+\chi}\otimes_{\mathcal{D}_{c+\chi}}\bullet$ is an equivalence $\operatorname{Coh}(\mathcal{D}_c)\xrightarrow{\sim}
\operatorname{Coh}(\mathcal{H}_c)$ inverse to $M\mapsto eM$. Note also that
$$e(\mathcal{E}_{c+\chi,-\chi}\otimes_{\mathcal{H}_{c+\chi}}
\mathcal{P}_{c+\chi}\otimes_{\mathcal{D}_{c+\chi}}\bullet)\cong \mathcal{D}_{c+\chi,-\chi}
\otimes_{\mathcal{D}_{c+\chi}}\bullet.$$
It follows that $M\mapsto eM: \operatorname{Coh}(\mathcal{H}_c)\rightarrow \operatorname{Coh}(\mathcal{D}_c)$
is a category equivalence.
\end{proof}

\begin{proof}[Proof of Theorem \ref{Thm:der_equiv2}]
Theorem \ref{Thm:der_equiv2} follows now from Propositions \ref{Prop:spher_equiv_local}
and Corollary \ref{Cor:quant_McKay}.
\end{proof}

\begin{proof}[Proof of Theorem \ref{Thm:der_equiv1}]
Thanks to Theorem \ref{Thm:der_equiv2}, it is enough to show that $\operatorname{Coh}(\mathcal{D}_c)
\xrightarrow{\sim} \operatorname{Coh}(\mathcal{D}_{c+\chi})$ for any $\chi\in \operatorname{Pic}(X^{reg})$
(so far, we know an equivalence for $\chi\in \operatorname{Pic}(X)$). Let $\mathcal{O}^{reg}(\chi)$ be a line
bundle on $X^{reg}$ corresponding to $\chi$. Since  $H^i(X^{reg},\mathcal{O}_{X^{reg}})=0$
for $i=1,2$, the bundle $\mathcal{O}^{reg}(\chi)$ admits a unique quantization to a $\mathcal{D}^{reg}_{c+\chi}$-$\mathcal{D}^{reg}_{c}$-bimodule that we denote by $\mathcal{D}^{reg}_{c,\chi}$.
Set $\mathcal{D}_{c,\chi}:=i_*\mathcal{D}^{reg}_{c,\chi}$, where $i$ denotes the inclusion of
$X^{reg}$ into $X$. We claim that $\mathcal{D}_{c,\chi}, \mathcal{D}_{c+\chi,-\chi}$ are mutually inverse
Morita equivalence bimodules. Indeed, we have natural bimodule homomorphisms
$$\mathcal{D}_{c,\chi}\otimes_{\mathcal{D}_{c}}\mathcal{D}_{c+\chi,-\chi}\rightarrow
\mathcal{D}_{c+\chi}, \mathcal{D}_{c+\chi,-\chi}\otimes_{\mathcal{D}_{c+\chi}}\mathcal{D}_{c,\chi}\rightarrow
\mathcal{D}_{c}.$$
They becomes an isomorphism after restriction to $X^{reg}$. So their kernels and cokernels are supported on $X^{sing}$.
By Corollary \ref{Cor:Conj_equiv_stat}, they are zero.
\end{proof}

\section{Applications}
Let us discuss some applications of Theorems \ref{Thm:simplicity}, \ref{Thm:der_equiv2}.
\subsection{Generalized Bernstein inequality}
Let $Y$ be a conical symplectic singularity.
Suppose that $\mathcal{D}_\lambda\widehat{\otimes} \mathcal{D}_\mu$ is simple for all $\lambda,\mu$
(that is true for $Y=V/\Gamma$ because $Y\times Y=(V\oplus V)/(\Gamma\times \Gamma)$ is again
a symplectic quotient singularity).

\begin{Prop}\label{Prop:HC_fin}
Any HC $\A_\lambda$-bimodule has finite length.
\end{Prop}
\begin{proof}
The proof repeats that from \cite[Section 4.3]{B_ineq}: note that we can prove the finiteness
of length for $\mathcal{D}_\lambda\widehat{\otimes} \mathcal{D}_\lambda^{opp}$ modules supported on
$X\times_Y X^-$ (where $X^-$ is the terminalization of $X$ corresponding to the cone opposite to
that of $X$) using the characteristic cycle argument because of Corollary
\ref{Cor:Conj_equiv_stat}.
\end{proof}

\begin{Cor}\label{Cor:HC_fin}
For any $c\in \param$, the regular $H_c$-bimodule has finite length.
\end{Cor}
\begin{proof}
It was explained in \cite[Remark 4.5]{B_ineq} that the finiteness of length will follow if we show that $H_c$ contains a minimal ideal of finite codimension. In the case, when $c$ is spherical (i.e., $H_ceH_c=H_c$) the existence
of a minimal ideal of finite codimension follows from Proposition \ref{Prop:HC_fin}
applied to the regular $eH_ce$-bimodule. In the general case, we can find $\chi\in \underline{\param}_{\Z}$
such that $c+\chi$ is spherical and abelian localization holds for $c+\chi$ and some choice of $X$.
Then we can use the derived functor $R\Hom_{H_{c+\chi}}(\Gamma(\mathcal{E}_{c,\chi}),\bullet)$ similarly to what was done in \cite[Section 4.3]{B_ineq} to show that $H_c$ has a minimal ideal of finite codimension.
\end{proof}

As was noted in the introduction, Corollary \ref{Cor:HC_fin} implies Theorem
\ref{Thm:B_ineq}. An analog of this theorem also follows for the algebras
$\A_\lambda$, where $Y$ satisfies the additional assumption in the beginning of
this section.

\subsection{Perverse equivalences}\label{SS_Perv}
Here we assume that $Y$ is such that the formal slices to all symplectic leaves are conical. Moreover,
we assume that the quantization $\mathcal{D}_\lambda$ is simple for all $\lambda$ (equivalently, for
a Weil generic $\lambda$). Then we can generalize results of \cite[Section 3.1]{Perv}
to this setting.

Namely, let us choose two chambers $C,C'$ in $\h$ that are opposite with respect to a common
face. Let us pick parameters $\lambda,\lambda'\in \h$
such that
\begin{enumerate}
\item
abelian localization holds for $(\lambda',C')$ and $(\lambda,C)$,
\item $\lambda-\lambda'$ lies in $\h_\Z^{X,X'}$, the intersection of the images of $\operatorname{Pic}(X),\operatorname{Pic}(X')$
in $\h$, where $X,X'$ are the $\Q$-terminalizations corresponding to the chambers $C,C'$.
\end{enumerate}
We can consider
the wall-crossing $\A_{\lambda'}$-$\A_\lambda$-bimodule $\A_{\lambda,\chi}$ defined as
in \cite[Section 6.3]{BPW}.

\begin{Prop}\label{Prop:perv}
After suitably modifying $\lambda,\lambda'$ (and hence $\chi$) by adding elements from
$\h_\Z^{X,X'}$  so that
abelian localization continues to hold for $(\lambda,C),(\lambda',C')$,
the functor $\A_{\lambda,\chi}\otimes^L_{\A_\lambda}\bullet$ becomes a perverse
derived equivalence $D^b(\A_\lambda\operatorname{-mod})\rightarrow D^b(\A_{\lambda'}\operatorname{-mod})$,
where the filtrations by Serre subcategories on $\A_\lambda\operatorname{-mod},
\A_{\lambda'}\operatorname{-mod}$ making these equivalences perverse are introduced
as in \cite[Section 3.1]{Perv}.
\end{Prop}
\begin{proof}
The proof basically repeats that of \cite[Theorem 3.1]{Perv}. The case when $\lambda$ is Weil generic
in an affine subspace with associated vector space $\operatorname{Span}_{\C}(C\cap C')$ is handled
using Corollary \ref{Cor:Conj_equiv_stat} (that in the setting of \cite{Perv}, where the case of smooth
$X$ is considered, is straightforward). The rest of the proof is the same as that of
\cite[Theorem 3.1]{Perv}.
\end{proof}


\subsection{Abelian localization and simplicity}\label{SS_ab_simpl}
Also let us mention two results that are proved as in \cite[Section 8.4]{BL} using the case
of $C'=-C$ in Proposition \ref{Prop:perv}. In that case, the filtration on $\A_\lambda\operatorname{-mod}$
is defined by $$\A_\lambda\operatorname{-mod}_i:=\{M\in \A_\lambda\operatorname{-mod}| \dim \VA(\A_\lambda/\operatorname{Ann}_{\A_\lambda}(M))\leqslant \dim X-2i\}.$$ We still keep the assumptions
of Section \ref{SS_Perv}.

\begin{Prop}\label{Prop:gen_simplicity}
There is a finite collection of hyperplanes in $\h$, each parallel to one of the walls in
$\h^{sing}$ with the following property: if $\lambda+\h_{\Z}$ does not intersect
the union of these hyperplanes, then $\A_\lambda$ is simple.
\end{Prop}

The same holds for the algebras $H_c$, this strengthens \cite[Theorem 4.2.1]{sraco}.
We also have a stronger version of the abelian localization theorem,
see Section \ref{SS_ab_loc}.

\begin{Prop}\label{Prop:ab_loc_strong}
Let $\lambda\in \param$. Then there is $\lambda_0\in \lambda+\h_{\Z}$
with the following property: for any $\lambda_1\in \lambda_0+(C\cap \h_{\Z})$ the functor $\Gamma_{\lambda_1}$
is a category equivalence.
\end{Prop}

%

\end{document}